\documentclass[reqno,10pt]{amsart}
\usepackage{amscd,amssymb}
\usepackage{latexsym}
\usepackage[all]{xy}

\def\into{\hookrightarrow}

\def\rra{\rightrightarrows}

\def\longto{\dashrightarrow}

\def\toisom{\widetilde{\to}}

\def\:{{\colon}}
\def\.{,\dots ,}
\def\wt{\widetilde}
\def\wh{\widehat}
\def\ol{\overline}

\def\Mor{{\rm Mor}}

\def\Ob{{\rm Ob}}

\def\Var{{\rm Var}}

\def\Ext{{\rm Ext}}

\def\Spf{{\rm Spf}}

\def\Ann{{\rm Ann}}

\def\Inv{{\rm Inv}}
\def\Spec{{\rm Spec}}

\def\Bl{{\rm Bl}}

\def\hatBl{{\rm \wh{Bl}}}

\def\unr{{\rm unr}}
\def\aff{{\rm aff}}

\def\dim{{\rm dim}}

\def\reg{{\rm reg}}
\def\rig{{\rm rig}}
\def\ad{{\rm ad}}

\def\an{{\rm an}}

\def\sing{{\rm sing}}

\def\Id{{\rm Id}}

\def\RHom{{\rm RHom}}
\def\Der{{\rm Der}}

\def\small{{\rm small}}

\def\inv{{\rm inv}}

\def\bfA{{\bf A}}

\def\bfF{{\bf F}}

\def\bfH{{\bf H}}

\def\bfL{{\bf L}}

\def\bfN{{\bf N}}

\def\bfQ{{\bf Q}}

\def\bfZ{{\bf Z}}

\def\gtC{{\mathfrak C}}

\def\gtI{{\mathfrak I}}
\def\gtJ{{\mathfrak J}}

\def\gtS{{\mathfrak S}}

\def\gtX{{\mathfrak X}}
\def\gtY{{\mathfrak Y}}
\def\gtZ{{\mathfrak Z}}

\def\gtf{{\mathfrak f}}

\def\calF{{\mathcal F}}

\def\calI{{\mathcal I}}
\def\calJ{{\mathcal J}}

\def\calO{{\mathcal O}}

\def\calX{{\mathcal X}}

\def\oT{{\ol T}}

\def\oX{{\ol X}}

\def\of{{\ol f}}
\def\og{{\ol g}}
\def\oh{{\ol h}}

\def\ox{{\ol x}}

\def\tilT{{\wt T}}

\def\tilW{{\wt W}}
\def\tilX{{\wt X}}

\def\hatA{{\wh A}}
\def\hatB{{\wh B}}
\def\hatC{{\wh C}}

\def\hatV{{\wh V}}

\def\hatX{{\wh X}}
\def\hatY{{\wh Y}}

\def\hatf{{\wh f}}

\def\ocalI{{\ol\calI}}

\def\ogtC{{\ol\gtC}}

\def\hatcalF{{\wh\calF}}

\def\hatcalI{{\wh\calI}}

\def\hatcalO{{\wh\calO}}

\def\alp{{\alpha}}

\def\veps{\varepsilon}
\def\ve{\veps}

\def\R+*{{\bf R^*_+}}

\newtheorem{theor}{Theorem}[subsection]
\newtheorem{prop}[theor]{Proposition}
\newtheorem{lem}[theor]{Lemma}
\newtheorem{cor}[theor]{Corollary}

\theoremstyle{definition}

\newtheorem{defin}[theor]{Definition}
\newtheorem{rem}[theor]{Remark}

\def\small{{\rm small}}

\def\hatVar{{\mathrm {\wh{Var}}}}

\def\QE{{\mathcal {QE}}}

\begin{document}

\title[Functorial desingularization over $\bfQ$: the non-embedded case]{Functorial desingularization
of quasi-excellent schemes in characteristic zero: the non-embedded case}
\author{Michael Temkin}
\thanks{The author wishes to thank E. Bierstone, P. Milman, J. Kollar, L. Illusie, O. Gabber and O. Villamayor for useful discussions, and the anonymous referee for numerous valuable comments.}
\address{\tiny{Einstein Institute of Mathematics, The Hebrew University of Jerusalem, Giv'at Ram, Jerusalem, 91904, Israel}}
\email{\scriptsize{temkin@math.huji.ac.il}}

\begin{abstract}
We prove that any reduced noetherian quasi-excellent scheme of characteristic zero admits a strong desingularization which is functorial with respect to all regular morphisms. As a main application, we deduce that any reduced formal variety of characteristic zero admits a strong functorial desingularization. Also, we show that as an easy formal consequence of our main result one obtains strong functorial desingularization for many other spaces of characteristic zero including quasi-excellent stacks and formal schemes, and complex or non-archimedean analytic spaces. Moreover, these functors easily generalize to non-compact setting by use of generalized convergent blow up sequences with regular centers.
\end{abstract}

\keywords{Resolution, singularities, quasi-excellent, functorial non-embedded desingularization}

\maketitle

\section{Introduction}

\subsection{Motivation}

\subsubsection{Historical overview}
Encouraged by Hironaka's work \cite{Hir} on resolution of singularities, Grothendieck introduced quasi-excellent (or {\em qe}) schemes in \cite[$\rm IV_2$, \S7.8]{ega} in order to provide a natural general framework for desingularization. Grothendieck observed that the schemes studied by Hironaka were schemes of finite type over a
local qe scheme $k$, and proved that if any integral scheme of finite type over a base scheme $k$ admits a desingularization in the weakest possible sense then $k$ is quasi-excellent. Grothendieck conjectured in \cite[$\rm IV_2$, 7.9.6]{ega} that the converse is probably true, and thus any qe scheme admits a desingularization, and claimed without proof that the conjecture holds true for noetherian qe schemes over $\bfQ$ as can be proved by Hironaka's method. The latter claim was never checked, so desingularization of qe schemes of characteristic zero remained a conjecture until the author's work \cite{temdes} in 2008.

It was shown in \cite{temdes} that, indeed, desingularization of qe schemes over $\bfQ$ can be obtained from Hironaka's results relatively easily. Perhaps this is close to the reduction Grothendieck had in mind, see Remark \ref{grem}. However, the main result of \cite{temdes} was that desingularization of qe schemes over $\bfQ$ follows from desingularization of varieties (by a more complicated argument). This is important, since many relatively simple proofs for varieties are now available, and they establish a canonical resolution which is compatible with all smooth morphisms between varieties, as opposed to the non-constructive original Hironaka's method.

\subsubsection{New goals}
The main disadvantage of the results from \cite{temdes} was that the desingularization results for qe schemes were weaker than their analogs for varieties in two aspects: (a) the centers of resolving blow ups were not regular, (b) no functoriality/canonicity was achieved. The aim of this paper is to strengthen the methods of \cite{temdes} in order to construct a desingularization by blowing up only regular centers and so that the whole blow up sequence is functorial in all regular morphisms (in particular, we cannot rely on the method of \cite{Hir} anymore since its resolution is not functorial). Our main results are Theorems \ref{desth} and \ref{hyperth} providing non-embedded desingularization of noetherian qe schemes over $\bfQ$ and non-compact objects including all qe (formal) schemes over $\bfQ$ and analytic spaces in characteristic zero. In comparison with the known non-embedded desingularization of varieties our results give the strongest known desingularization with the exception mentioned in Remark \ref{redrem}: the centers $V_i\into X_i$ of our blow ups are not contained in the maximum locus of the Hilbert-Samuel (or HS) function, moreover, $X_i$ does not have to be normally flat along $V_i$. As a drawback, our method does not treat non-embedded desingularization of generically non-reduced schemes.

Although our method follows the method of \cite{temdes} closely, we tried to make the paper as self-contained as possible. So, the familiarity with \cite{temdes} can be useful but is not necessary. In particular, we give/recall all needed definitions. As opposed to \cite{temdes}, this time we split the exposition into two separate papers concerning non-embedded and embedded desingularization. Note that unlike the case of varieties, we have to deal with both cases because one cannot deduce the non-embedded case from the embedded one -- there are qe schemes that cannot be locally embedded into regular ones. This paper deals only with the non-embedded case (without boundaries), while the embedded case (with boundaries) is dealt with in \cite{emb}.

\subsubsection{Applications}
Currently, it seems that the most applicable case of desingularization of general qe schemes over $\bfQ$ is when the schemes are of finite type over a ring $k[[t_1\. t_n]]$ with $k$ a field. For example, this case (in its non-functorial version from \cite{temdes}) was used for studying log-canonical thresholds in \cite{FEM}, motivic integration in \cite{Nic}, and desingularization of meromorphic connections in \cite{Ked}. The non-functorial method of \cite{temdes} could only desingularize affine qe formal schemes, and in my communication with Mustata, Nicaise, Kedlaya and Soibelman I was urged to strengthen the method so that it may desingularize arbitrary formal schemes of finite type over $k[[t_1\. t_n]]$. This served as the main motivation to continue the research on desingularization of qe schemes and I am very grateful to them all for the encouragement. I would expect that desingularization of formal schemes of finite type over $k[[t_1\. t_n]]$ should be the most applicable particular case of the main results of this paper. Also, desingularization of Berkovich analytic spaces of characteristic zero seems to be new (at least, in the non-good case).

\subsection{Main result}\label{mainsec}

\subsubsection{Strong non-embedded desingularization}
We refer to \S\S\ref{schsec},\ref{blowsec} for the terminology, which includes the notions of quasi-excellent schemes, regular loci and morphisms, and blow up sequences. Here is the strong non-embedded desingularization theorem for generically reduced noetherian quasi-excellent schemes in characteristic zero. In this paper "strong" means that the desingularization blows up only regular subschemes, unlike the terminology of Bierstone-Milman, where one also requires that each center lies in the maximum locus of the HS function.

\begin{theor}\label{desth}
For any noetherian quasi-excellent generically reduced scheme $X=X_0$ over $\Spec(\bfQ)$ there exists a blow up sequence $\calF(X)\:X_n\longto X_0$ such that the following conditions are satisfied:

(i) the centers of the blow ups are disjoint from the preimages of the regular locus $X_\reg$;

(ii) the centers of the blow ups are regular;

(iii) $X_n$ is regular;

(iv) the blow up sequence $\calF(X)$ is functorial with respect to all regular morphisms $X'\to X$, in the sense that $\calF(X')$ is obtained from $\calF(X)\times_XX'$ by omitting all empty blow ups.
\end{theor}

\subsubsection{On the non-reduced case}
\begin{rem}\label{redrem}
(i) Theorem \ref{desth} implies that the same claim holds for all noetherian qe schemes over $\bfQ$. Indeed, given an arbitrary such $X$ with reduction $\tilX$ we can consider the blow up sequence $i_*\calF(\tilX)\:X'\longto X$ which is the pushforward of $\calF(\tilX)\:\tilX'\longto\tilX$ with respect to the closed immersion $i\:\tilX\into X$ (see \S\ref{pushsec2}). Then the reduction of $X'$ is $\tilX'$, so it is regular. Next we kill all generically non-reduced components by blowing them up along their reductions, obtaining a blow up $X''\to X'$ with generically reduced $X''$. Finally, it remains to apply $\calF(X'')\:X'''\longto X''$ to construct a functorial desingularization $X'''\longto X$.

(ii) We ignore on purpose the case when $X$ is not reduced along an irreducible component because in this case the assertion of the theorem is not the "right" version of desingularization: such naive desingularization simply destroys all generically non-reduced components. In addition, existence of naive desingularization of non-reduced schemes does not assert anything new, as we saw in (i).

(iii) The correct way to desingularize non-reduced schemes is to make them normally flat along the reduction. Loosely speaking, such desingularization does not kill non-reduced components, but makes their nilpotent structure "as smooth as possible". The only known way to achieve such desingularization is to use an algorithm such that each its center $V_i$ lies in the maximum HS locus of $X_i$ (or, at least, $X_i$ is normally flat along $V_i$), and to stop the algorithm just before it blows up the non-reduced components. Our algorithm does not possess this property (see also Remark \ref{HSrem}).
\end{rem}

\subsection{Overview}

\subsubsection{The black box strategy}
Now let us discuss our method and the structure of the paper. In order to construct a desingularization $\calF$ of all qe schemes over $\bfQ$ we use as a black box any algorithm $\calF_\Var$ which desingularizes varieties of characteristic zero and is functorial with respect to all regular morphisms (i.e. satisfies the conditions (i), (iii) and (iv) from the theorem). Moreover, if $\calF_\Var$ is strong then we can also achieve that $\calF$ is strong, i.e. satisfies the condition (ii) as well. We consider both strong and non-strong cases mainly because this does not cost us any extra-work; also, there exist many non-strong algorithms, e.g. those of \cite{Kol} and \cite{Wl}.

\subsubsection{Extending $\calF_\Var$}\label{extsec}
It was checked in a very recent work \cite{bmt} that the algorithm of Bierstone-Milman is functorial with respect to all regular morphisms (not necessarily of finite type). Thus, for the sake of concreteness, we can start with the algorithm $\calF_\Var$ from \cite{bmfun}. (Alternatively, it is shown in \cite{bmt} that one can start with any algorithm $\calF_\bfQ$ for varieties over $\bfQ$ and extend it to all varieties using approximation results of \cite[$\rm IV_3$, \S8.8]{ega}. In such case, one can also build on many other available algorithms.) In \S\ref{setup} we fix our terminology and recall many basic facts about blow ups, desingularizations and formal schemes. We warn the reader that the terminology differs in part from that of \cite{temdes}, including blow ups from \S\ref{blsec}, desingularization from \S\ref{dessec} and qe formal schemes from \S\ref{qeforsec}.

We extend in \S\ref{locsec} the functor $\calF_\Var$ to pairs $(X,Z)$, where $X$ is a qe scheme over $\bfQ$ and $Z\into X$ is a Cartier divisor containing the singular locus of $X$ and isomorphic to a disjoint union of varieties. Each $\calF_\Var(X,Z)$ is a desingularization of $X$, and it remains unclear if it really depends on $Z$. Similarly to the method of \cite{temdes}, this construction goes by passing to the formal completion of $X$ along $Z$, thus obtaining a rig-regular formal variety $\gtX=\hatX_Z$, and algebraizing $\gtX$ by a variety $X'$. Then $\calF_\Var(X')$ induces desingularizations on $\gtX$ and $X$, and the main problem we have to solve is that the choice of $X'$ is absolutely non-canonical. Moreover, and this causes the main trouble, even the ground field of $X'$ can be chosen in many ways. This complication is by-passed by proving that all information about the singularities of $\gtX$ can be extracted already from a sufficiently thick nilpotent neighborhood $X_n\into\gtX$ of its closed fiber $\gtX_s$. Moreover, everything is determined by the scheme $X_n$, and is independent of a variety structure of $X_n$, which is not unique. In particular, we prove in \S\ref{intrsec} that $\calF_\Var(\gtX):=\wh{\calF_\Var(X')}$ is canonically defined already by $X_n$, and is therefore independent of the choice of the algebraization $X'$. This section is the heart of the paper, and it is the main novelty since \cite{temdes}. As in loc.cit., Elkik's results from \cite{Elk} are the main tool we are using for algebraization. It will also be convenient (though not critical) to use an improvement of some Elkik's results by Gabber-Ramero.

\begin{rem}\label{HSrem}
It is a natural question if one can strengthen our method so that it will also preserve the property that the centers lie in the maximum HS loci. It seems that the only stage of our method that will not work is algebraization. At this stage one would obtain a formal variety $\gtX$ such that the maximum locus of its HS function is given by an open ideal (informally speaking, the HS function drops on its generic fiber $\gtX_\eta$). Unfortunately, this is clearly insufficient to ensure that $\gtX$ is locally algebraizable, so it is not clear how to use desingularization of varieties to desingularize such $\gtX$.
\end{rem}

\subsubsection{Constructing $\calF$}
In \S\ref{fsec} we use the desingularizations $\calF_\Var(X,Z)$ to construct another desingularization functor $\calF$ which applies to all generically reduced noetherian qe schemes, and this is done by induction on codimension similarly to the argument from \cite[2.3.4]{temdes}. This time we must work more carefully in order not to lose functoriality of the algorithm and regularity of the centers, but the basic idea is the same. Actually, we construct a sequence of functors $\calF^{\le d}$ which desingularize a qe scheme $X$ over the set $X^{\le d}$ of points of codimension at most $d$. The construction is inductive: given the blow up sequence $\calF^{\le d-1}(X)\:X'\to X$ we insert new blow ups which resolve its centers and the source over few "bad" points of codimension $d$. Actually, after localization at the bad points we are dealing with schemes whose singular loci are disjoint unions of varieties, so the functor $\calF_\Var$ suffices to construct $\calF^{\le d}(X)$ by patching in $\calF^{\le d-1}(X)$ over the bad points.

\begin{rem}\label{difrem}
(i) Because of functoriality, we must build $\calF$ from scratch, so it differs from $\calF_\Var$ even for simple varieties. See an expository survey \cite[\S5.2]{survey}, where the action of these algorithms is computed and compared for few simple examples of varieties.

(ii) It remains to be an interesting open question how far one can push the standard desingularization functors $\calF_\Var$ defined in terms of derivative ideals, including the question if one can extend $\calF_\Var$ to all qe schemes of characteristic zero.

(iii) The author conjectures that the existing algorithms can be extended (up to minor modifications) to all schemes over $\bfQ$ which locally admit closed immersions into regular affine schemes $U=\Spec(A)$ with {\em enough derivatives}, and that the extended algorithm is compatible with all regular morphisms. The condition on
derivatives means that the sheaf $\Der_{U/\bfZ}$, which does not have to be quasi-coherent and can even have zero stalks, admits global sections $\partial_1\.\partial_n$ such that for any point $u\in U$ the images of $\partial_i$'s generate the tangent space $(m_u/m_u^2)^*$. Since existence of the above derivatives is equivalent to existence of an (analog of) Taylor homomorphism $A\to A[[T]]$ due to Bierstone-Milman, this conjecture agrees with Bierstone-Milman philosophy. It is an interesting problem to generalize the algorithm from \cite{bmfun} accordingly to the conjecture.

(iv) The above conjecture implies functorial desingularization of formal varieties of characteristic zero. In particular, it would drastically simplify our work in this paper by proving a much more general result on formal desingularization than we prove in \S\ref{formdesth}, and such approach would not use Elkik's results (we
desingularize all qe formal schemes over $\bfQ$ in \S\ref{catsec}, but this is based on the intermediate result of \S\ref{formdesth}).
\end{rem}

\begin{rem}\label{grem}
The author is grateful to O. Villamayor for pointing out that a very similar scheme of induction on codimension was used by Hironaka himself already in his great work of 1964, see \cite[\S4.1]{Hir}, where this method is called "localization". In \cite[Th. 2.3.6]{temdes} induction on codimension was used to deduce desingularization of all noetherian qe schemes over $\bfQ$ from Hironaka's desingularization theorem. Thus, it seems plausible that this proof is close in spirit to the argument Grothendieck talked about in \cite[$\rm IV_2$, 7.9.6]{ega}. Note, however, that Grothendieck also made a stronger assertion that desingularization of arbitrary noetherian qe schemes follows from the case of spectra of noetherian complete local rings. This claim still seems a mystery, and I do not know if this is true and how one can prove this (e.g. by localization technique).
\end{rem}

\subsubsection{Applications to other categories}
Finally, in \S\ref{catsec} we use the functorial desingularization of qe schemes over $\bfQ$ to establish desingularization in other categories. We show that our result implies desingularization of stacks over $\bfQ$, of complex/rigid/Berkovich analytic spaces of characteristic zero, and of qe formal schemes over $\bfQ$. The results about stacks (not covered by varieties) and formal schemes are new. Desingularization of blow ups of affine formal schemes over $\bfQ$ was proved in \cite[4.3.2]{temdes}, but it was impossible to globalize this result without canonicity of the construction, see \cite[4.3.3]{temdes} and the remark after it. The analytic desingularization is well known -- it is done by absolutely the same method as its algebraic analog, but strictly speaking it required a parallel proof until now. We show that it is a formal consequence of desingularization of qe schemes. Since the latter is obtained in  this paper from a desingularization $\calF_\bfQ$ of $\bfQ$-varieties, we see that all desingularization theories can be built in a very algebraic way using a single algorithm $\calF_\bfQ$. At first glance, this might look as a sharp surprise in view of existence of non-algebraizable formal and analytic singularities. Finally, using canonical desingularization we also desingularize non-compact objects in \S\ref{noncomsec}, including locally noetherian qe schemes over $\bfQ$. In particular, this settles completely Grothendieck's conjecture in characteristic zero, as non quasi-compact qe schemes were not treated in \cite{temdes}.

\tableofcontents

\section{Setup}\label{setup}

\subsection{Schemes and morphisms}\label{schsec}

\subsubsection{Varieties}
{\em Variety} or {\em algebraic variety} in this paper always means a scheme $X$ which admits a finite type morphism $X\to\Spec(k)$ to the spectrum of a field. If such a morphism is fixed then we say that $X$ is a $k$-variety and $k$ is the {\em ground field} of $X$. The reader should be aware that an abstract scheme $X$ may admit many different structures of an algebraic variety (especially, when $X$ is not reduced).

\subsubsection{Ideals and closed subschemes}
Given a scheme $X$, we will freely pass between closed subsets and reduced closed subschemes of $X$. Also, we will freely pass between ideals $\calI\subset\calO_X$ and closed subschemes $Z\into X$.

\subsubsection{Pro-open subsets and subschemes}
A {\em pro-open subset} of a scheme $X$ is a subset $S\subset|X|$ which coincides with the intersection of all its neighborhoods. An equivalent condition is that $S$ is closed under generalizations. If $(S,\calO_X|_S)$ is a scheme then we call it the {\em pro-open subscheme} corresponding to $S$. For example, if $x\in X$ then
$X_x=\Spec(\calO_{X,x})$ is a pro-open subscheme of $X$.

\subsubsection{Filtration by codimension}
For a locally noetherian scheme $X$ by $X^{<d}$, $X^{\ge d}$, etc., we denote the subsets of $X$ consisting of all points of codimension strictly less than $d$, larger than $d$, etc. We note that $S=X^{<d}$ is a pro-open subset, but usually it does not underly a pro-open subscheme. We will use in the sequel {\em filtration by codimension} $\emptyset=X^{<0}\subset X^{<1}\subset X^{<2}\subset\dots$, which however may be infinite.

\begin{rem}
Even if $X$ is noetherian its dimension can be infinite, due to some pathological examples by Nagata. For this reason, one should use noetherian induction instead of a more naive induction by dimension. In all cases of practical interest, however, noetherian schemes are finite dimensional.
\end{rem}

\subsubsection{Schematical closure}\label{schemsec}
If $U$ is a pro-open subscheme in a locally noetherian scheme $S$ and $Z_U\into U$ is a closed subscheme then by the {\em schematical closure} of $Z_U$ in $S$ we mean the schematical image $Z$ of the morphism $i\:Z_U\to S$ (i.e. the minimal closed subscheme of $S$ such that $i$ factors through it). The following lemma indicates that this construction works as fine as in the case when $U$ is open in $S$, and $Z$ is actually the minimal extension of $Z_U$ to a closed subscheme in $S$.

\begin{lem}\label{schemlem}
Keep the above notation, then $Z_U$ coincides with the restriction of $Z$ on $U$.
\end{lem}
\begin{proof}
It suffices to prove that $Z_U$ admits any extension to a closed subscheme $Z\into S$ since the minimal extension then exists by local noetherianity of $S$. An arbitrary extension was constructed in \cite[2.1.1]{temdes} as follows: first one uses \cite[$\rm IV_4$, 8.6.3]{ega} to extend $Z_U$ to a closed subscheme $Z_V$ of a sufficiently small open neighborhood $V$ of $U$, and then one extends $Z_V$ to $S$ using \cite[6.9.7]{egaI}.
\end{proof}

\subsubsection{Schematical density}
Assume that $S$ is a scheme with a pro-open subset $U$. We say that $U$ is {\em schematically dense} in $S$ if for any proper closed subscheme $S'\into S$ there exists $u\in U$ such that $S'\times_SS_u$ is a proper subscheme of $S_u=\Spec(\calO_{S,u})$ (possibly empty). In particular, if $U$ is a pro-open subscheme then this condition agrees with the usual definition that the schematic image $V$ of the embedding $U\into S$ coincides with $S$. Indeed, $V$ is the minimal closed subscheme of $S$ whose restriction onto each $S_u$ with $u\in U$ coincides with $S_u$. So, $V=S$ if and only if any proper closed subscheme $S'\into S$ restricts to a proper closed subscheme of some $S_u$ with $u\in U$.

\subsubsection{$U$-admissibility}\label{udensesec}
An $S$-scheme $X$ is called {\em $U$-admissible} if the preimage of $U$ in $X$ is schematically dense. This follows the terminology of \cite{RG} and we will not use the notion "admissible" in other meanings.

\subsubsection{Regular morphisms}
Following \cite[$\rm IV_2$, 6.8.1]{ega} we call a morphism of schemes $f\:Y\to X$ {\em regular} if it is flat and has geometrically regular fibers (in particular, the fibers are locally noetherian). For morphisms of finite type regularity is equivalent to smoothness, so it can be viewed as a generalization of smoothness to morphisms not necessarily of finite type. A homomorphism of rings $f\:A\to B$ is {\em regular} if $\Spec(f)$ is regular.

\begin{rem}
(i) We remark that Hironaka uses the notion of universally regular morphisms instead of regular morphisms reserving the notion of regularity to what we call reg morphisms below. However, our definition of regularity is the standard one.

(ii) Let us temporarily say that a morphism $f\:Y\to X$ is {\em reg} if it is flat and has regular fibers. Then a morphism is regular in our sense if and only if it is universally reg.

(iii) It seems that reg morphisms are not worth a study (and a special name). For example, one can easily construct an example of a flat family of integral curves over $\bfA^1_{\bfF_p}$ such that the general fiber is regular but all closed fibers are singular.
\end{rem}

\subsubsection{Singular locus}
We define the {\em regular locus} $X_\reg$ of a scheme $X$ as a set of points at which $X$ is regular, and the {\em singular locus} $X_\sing$ is defined as the complement of $X_\reg$.

\subsubsection{Compatibility with regular morphisms}\label{compsec}
It is well known that regular/singular locus is compatible with regular morphisms, i.e. for a regular morphisms $f\:Y\to X$ we have that $Y_\sing=f^{-1}(X_\sing)$ and $Y_\reg=f^{-1}(X_\reg)$. See, for example, \cite[Th. 51]{Mat}.

\subsubsection{Singular locus of a morphism}
By the {\em singular locus} $g_\sing$ of a morphism $g\:Y\to X$ we mean the set of points $y\in Y$ at which $g$ is not regular (i.e. $g_y\:\Spec(\calO_{Y,y})\to\Spec(\calO_{X,g(y)})$ is not regular).

\begin{lem}\label{complem}
(i) If $f\:Z\to Y$ and $g\:Y\to X$ are morphisms of schemes and $f$ is regular, then $(g\circ f)_\sing=f^{-1}(g_\sing)$.

(ii) If $g\:A\to B$ and $f\:B\to C$ are local homomorphisms between local rings and $f$ is regular, then $g$ is regular if and only if $f\circ g$ is regular.
\end{lem}
\begin{proof}
Obviously, (ii) is a particular case of (i). On the other hand, the claim that $z\in(g\circ f)_\sing$ if and only if $z\in f^{-1}(g_\sing)$ reduces to checking (ii) for the local homomorphisms $g_y\:\calO_{X,x}\to\calO_{Y,y}$ and $f_z\:\calO_{Y,y}\to\calO_{Z,z}$, where $y=f(z)$ and $x=g(y)$. Since $f_z$ is faithfully flat, $g_y$ is flat if and only if $f_z\circ g_y$ is flat. The homomorphism $f_z\otimes_{\calO_{X,x}}k(x)\:\calO_{Y,y}/m_x\calO_{Y,y}\to\calO_{Z,z}/m_x\calO_{Z,z}$ is the base change of
$f_z$, hence it is regular, and by \S\ref{compsec} we obtain that $\calO_{Y,y}/m_x\calO_{Y,y}$ is regular if and only if $\calO_{Z,z}/m_x\calO_{Z,z}$ is regular. This concludes the proof.
\end{proof}

\subsubsection{Quasi-excellent schemes}
For shortness, we will abbreviate the word quasi-excellent as {\em qe}. Since qe schemes are defined by two conditions which are of an interest of their own, we introduce the corresponding classes of schemes. We say that $X$ is an {\em N-scheme} if it is locally noetherian and for any $Y$ of finite type over $X$ the regular locus $Y_\reg$ is open. We say that $X$ is a {\em G-scheme} if for any point $x\in X$ the completion homomorphism $\calO_{X,x}\to\hatcalO_{X,x}$ is regular. It was proved by Grothendieck that it suffices to check this only at the closed points. A scheme is {\em qe} if it is both G and N scheme. If in addition $X$ is universally catenary then it is called {\em excellent}. We say that a ring is G, N, qe or excellent if its spectrum is so. We list few basic well known properties of qe schemes.

\begin{lem}
(i) If $X$ is N, G, qe or excellent then any $X$-scheme of finite type is so.

(ii) $X$ is a $G$-scheme if and only if the local rings of all closed points $x\in X$ are G-rings.

(iii) If $X=\Spec(A)$ is an affine G-scheme then for any ideal $I\subset A$ the completion homomorphism $A\to\hatA_I$ is regular.
\end{lem}

\subsubsection{Quasi-excellence and completions}
The G-property can be lost when passing to formal completions. Fortunately, the situation with quasi-excellence is better, as the following result shows.

\begin{theor}[Nishimura-Nishimura]\label{NNth}
Let $A$ be a noetherian ring containing $\bfQ$ and let $I\subset A$ be an ideal such that $A$ is complete in the $I$-adic topology. Then $A$ is qe if and only if $A/I$ is qe. In particular, formal completion of $\bfQ$-algebras preserves quasi-excellence.
\end{theor}

\begin{rem}\label{NNrem}
This is Theorem B in \cite{NN}. The main ingredient in its proof is existence of weak resolution of singularities for local $A$-schemes of essentially finite type, and this tool imposes the restriction on the characteristic. In the original proof, Nishimura-Nishimura used Hironaka's theorem (which covers local qe schemes). Nowadays, one can use any desingularization of varieties and the results of \cite{temdes} as an alternative. (Modern desingularizations of varieties are much simpler than Hironaka's proof, and \cite{temdes} shows how any of them can be used as a black box to resolve arbitrary qe scheme over $\bfQ$.)
\end{rem}

Since qe schemes were introduced by Grothendieck, it took more than 20 years to prove \ref{NNth}. The question whether the characteristic zero assumption can be removed was open for another 20 years, until it was answered affirmatively by Ofer Gabber (unpublished).

\begin{theor}[Gabber]\label{gabth}
Let $A$ be a noetherian ring with an ideal $I$ such that $A$ is complete in the $I$-adic topology. Then $A$ is qe if and only if $A/I$ is qe. In particular, formal completion preserves quasi-excellence.
\end{theor}

\begin{rem}\label{gabrem}
(i) Gabber's theorem is not fully written down yet, but the main lines of the proof are outlined in Gabber's letter to Laszlo. Gabber accurately checks the proof of \cite{NN} and shows which modifications are needed in order to use only the following weaker desingularization ingredient. Weak local uniformization of qe schemes: any qe scheme can be covered by a regular one in the topology generated by alterations and flat quasi-finite coverings. The latter theorem is a subtle and important result of Gabber that will be published soon in a volume by Illusie, Laszlo and Orgogozo. Actually, it is the only desingularization result established for all qe schemes.

(ii) Since a full proof of Theorem \ref{gabth} is not available yet, we will make only a single conditional use of it in the paper, see \S\ref{formalsec}. All results in characteristic zero, including all our main theorems, are independent of \ref{gabth} and only make use of \ref{NNth}. Due to Remark \ref{NNrem}, this does not assume dependency on Hironaka's paper.
\end{rem}

\subsubsection{Categories}
In this paper, we denote by $\QE$ the category of generically reduced noetherian qe schemes and by $\QE_\reg$ its subcategory containing only regular morphisms. Of main interest for our needs are the full subcategories $\Var_{p=0,\reg}$ and $\QE_{p=0,\reg}$ of $\QE_\reg$. The objects in $\Var_{p=0,\reg}$ are finite disjoint unions of
generically reduced varieties of characteristic zero and the objects of $\QE_{p=0,\reg}$ are generically reduced noetherian qe schemes over $\bfQ$. The reason to consider disjoint unions of varieties defined over different fields will become clear in \S\ref{affsec}.

\subsection{Blow up sequences}\label{blowsec}

\subsubsection{Blow ups}\label{blsec}
Basic facts about blow ups can be found in \cite[\S2.1]{temdes} or in the literature cited there. Recall that the blow up $f\:\Bl_V(X)\to X$ along a closed subscheme $V$ is the universal morphism such that $f^*(V):=V\times_X\Bl_V(X)$ is a Cartier divisor. We will use the terminology of \cite{temdes} with one important exception: when we say that $Y\to X$ is a blow up (or $Y$ is a blow up of $X$), we always mean that the center $V$ of the blow up is fixed. So, strictly speaking, we always mean (usually implicitly) that an isomorphism $Y\toisom\Bl_V(X)$ is fixed. The reason for this change is that we will study functoriality, and the choice of $V$ increases canonicity of constructions, see, for example, \S\ref{pushsec}. Following the convention of \cite{Kol} we call a blow up $\Bl_{\emptyset}(X)\toisom X$ {\em empty} or {\em trivial} blow up. This is the only blow up we will sometimes ignore. However, even empty blow ups play important synchronizing role when patching local desingularizations, see \S\S\ref{funsubsec}--\ref{invsec}. Note that though any blow up along a Cartier divisor induces an isomorphism on the level of schemes, it may play a non-trivial role for functorial desingularization. In our case this is mainly the synchronization (similarly to empty blow up), but in the embedded case such blow ups induce non-trivial operations of strict and controlled transforms, so they cannot be ignored by no means.

\subsubsection{Blow up sequences}\label{blseqsec}
A price one has to pay for using a finer notion of blow ups is that the composition $X''\to X'\to X$ of blow ups cannot be considered as a blow up in a natural way: though $X''\to X$ is isomorphic to a blow up $\Bl_W(X)\to X$, it is not clear how to choose $W$ canonically. Therefore we define a {\em blow up sequence of length $n$} to be a composition of $n\ge 0$ blow ups $X_n\stackrel{f_n}\to\dots\stackrel{f_2}\to X_1\stackrel{f_1}\to X_0$ with centers $V_i\into X_i$. In particular, a blow up sequence of length $0$ is just the object $X_0$ itself. Such sequence is called {\em empty}. It will be convenient to denote a blow up sequence as $f\:X_n\longto X_0$ and say that $V_i$'s are its centers. We stress that a blow up sequence is not defined by the morphism $X_n\to X$ even when we blow up regular schemes along regular centers, see \cite[3.33]{Kol}. A sequence will be called {\em contracted} if all its blow ups are non-empty. Using dashed arrows we will sometimes split blow up sequences, e.g. $f\:X_n\longto X_{i+1}\to X_i\longto X_0$.

\subsubsection{Equality/isomorphism of blow up sequences}
An {\em isomorphism} between two blow up sequences $X'_n\longto X_0$ and $X_n\longto X_0$ is a set of isomorphisms $X'_i\toisom X_i$ which identify $X_0$'s and the centers. Obviously, if two blow up sequences of $X$ are isomorphic then the isomorphism is unique. For this reason, we will often say by slight abuse of language that two such blow up sequences are {\em equal}, and this cannot cause any confusion. For example, this agreement allows us to say "unique" in Lemma \ref{pushlem} below instead of "unique up to unique isomorphism".

\subsubsection{Trivial extension}\label{trivsec}
We say that a blow up sequence $\of\:\oX'\to X$ is a {\em trivial extension} of a blow up sequence $f\:X'\to X$ if $f$ is obtained from $\of$ by removing few empty blow ups (here we invoke our agreement about equality of blow up sequences).

\subsubsection{Flat base changes}\label{flatsec}
Any blow up sequence $f\:X_n\longto X_0$ is {\em compatible} with any flat base change $g\:Y_0\to X_0$ in the sense that $Y_1:=Y_0\times_{X_0} X_1$ is a flat $X_1$-scheme isomorphic to the blow up of $Y_0$ along $Y_0\times_{X_0}V_0$, and similarly for all further blow ups in the sequence. Naturally, the induced blow up sequence $Y_n\longto Y_0$ will be called the {\em pullback} of $f$ with respect to $g$ and will be denoted as $g^*(f)$.

\subsubsection{Pushing forward with respect to pro-open immersions}\label{pushsec}
One advantage of having the center $V$ of a blow up $f_U\:U'\to U$ fixed is that for any pro-open immersion $U\into X$ with a locally noetherian $X$ we can canonically extend $f_U$ to a blow up $f\:X'\to X$. We simply take the schematical closure of $V$ in $X$ to be the center of $f$. Note that $f_U=f\times_XU$ by Lemma \ref{schemlem} and \S\ref{flatsec}, in particular, $f$ extends $f_U$. Iterating the same construction we obtain the following easy result.

\begin{lem}\label{pushlem}
Assume that $f_U\:U'\longto U$ is a blow up sequence and $i\:U\into X$ is a pro-open immersion with a locally noetherian $X$. Then there exists a unique blow up sequence $f\:X'\longto X$ such that $f_U=f\times_X U$ and the centers of $f_U$ are schematically dense in the centers of $f$.
\end{lem}

The blow up sequence $f$ will be called the {\em pushforward} of $f_U$ with respect to the pro-open immersion $i$.

\subsubsection{$T$-supported blow up sequences}\label{supsec}
Assume that $X$ is an $S$-scheme, $U\into S$ is an open subscheme, $T=S\setminus U$ and $V=X\times_SU$. Then we say that a blow up sequence $X'\longto X$ is {\em $T$-supported} if its centers lie over $T$. In order to avoid any confusion with $V$-admissibility of the centers, we will not say that $f$ is "$V$-admissible" in this situation (as opposed to \cite{RG} and \cite{temdes}). Actually, we will often be interested in the following two extreme cases: (a) the centers of $f$ are $V$-admissible, and (b) the centers of $f$ are $T$-supported. Note that (a) takes place if and only if $f$ is the pushforward of its restriction $f\times_XV$ with respect to $i\:V\into X$, and (b) takes place if and only if $f\times_XV$ is a sequence of empty blow ups.

\subsubsection{Strict transform}\label{strictsec}
The following definitions and facts about blow up sequences follow from the well known particular case of usual blow ups (i.e. the sequences of length one), which can be found in \cite[\S1]{Con}. Given a closed subscheme $Z_0\into X_0$ we define the {\em strict transform} $Z_n=f^!(Z_0)$ of $Z_0$ under $f\:X_n\longto X_0$ as the iterative strict transform of $Z_0$ with respect to $f_i$'s. Note that $Z_n\longto Z_0$ is the blow up sequence whose centers are the scheme-theoretic preimages of the centers $V_i$ of $f$. More concretely, $f$ induces a blow up sequence $Z_n\longto Z_0$, where $Z_i=f_i^!(Z_{i-1})$ is isomorphic to the blow up of $Z_{i-1}$ along $V_{i-1}\times_{X_{i-1}}Z_{i-1}$.

\subsection{Desingularization}

\subsubsection{Desingularization of schemes}\label{dessec}
By {\em desingularization} of a scheme $X$ we mean an $X_\sing$-supported contracted blow up sequence $f\:X'\longto X$ with regular $X'$. We will see in \S\ref{compmorsec} why it is convenient to forbid empty blow ups.  Note also that although the definition makes sense for any locally noetherian scheme $X$, we will study only the case when $X$ is generically reduced. The reason for this was explained in Remark \ref{redrem}.

\subsubsection{Strong desingularization}
A desingularization of a scheme is called {\em strong} if the centers of its blow ups are regular schemes. We remark that most of the recent approaches based on order reduction of marked ideals lead to a desingularization which is not strong. In particular, this is the case for the algorithms from \cite{Wl} and \cite{Kol}, see \cite[3.106]{Kol} or \cite[8.2]{bmfun}. A strong desingularization for varieties of characteristic zero can be found in \cite{Hir}, \cite{BM}, or \cite{Vil}.

\subsubsection{Compatibility with morphisms}\label{compmorsec}
We say that desingularizations $g'\:Y'\longto X'$ and $g\:Y\longto X$ of $X'$ and $X$, respectively, are {\em compatible} with respect to a flat morphism $f\:X'\to X$ if $g\times_XX'$ is a trivial extension of $g'$. Since $g'$ is contracted, this happens if and only if $g'$ is obtained by removing all empty blow ups from $g\times_XX'$. The latter are precisely the pullbacks of blow ups whose centers are disjoint from $f(X')$. In particular, if $X_\sing\subset f(X')$, e.g. $f$ is surjective, then automatically $g'=g\times_XX'$. We will see in Lemma \ref{sublem} that this fact has simple but subtle and important consequences. Actually, we will only be interested in the case when $f$ is regular.

\subsubsection{Functorial desingularization}\label{funsubsec}
If $\gtC$ is a subcategory of $\QE_\reg$, then by a {\em functorial (strong) desingularization on $\gtC$} we mean a rule (or a blow up sequence functor) $\calF$ which to each object $X$ from $\gtC$ assigns a (strong) desingularization $\calF(X)\:\oX\longto X$ in a way compatible with the morphisms from $\gtC$, i.e. for any morphism $f\:X'\to X$ from $\gtC$ the desingularizations $\calF(X)$ and $\calF(X')$ are compatible with respect to $f$. Clearly, one can view functorial desingularization as a functor to an appropriate category of blow up sequences, but we do not need to develop such formalized approach. Nevertheless, it will be convenient to abuse the language and occasionally call $\calF$ a desingularization functor. Note also that a non-functorial desingularization corresponds to the case when $\Mor(\gtC)$ consists of identities only.

\begin{lem}\label{sublem}
Assume that $\gtC$ is closed under taking finite disjoint unions. If $\calF$ is a functorial desingularization on $\gtC$ and $f,g\:X'\to X$ are two morphisms in $\gtC$, then $f^*\calF(X)=g^*\calF(X)$ as blow up sequences (taking into account the empty blow ups).
\end{lem}
\begin{proof}
Since $\of=f\coprod\Id_X$ and $\og=g\coprod\Id_X$ are surjective regular morphisms from $\oX=X'\coprod X$ to $X$, we have that $\of^*\calF(X)=\calF(\oX)=\og^*\calF(X)$. Restricting this equality over $X'\into \oX$ gives the required equality.
\end{proof}

\begin{rem}
(i) Up to empty blow ups both $f^*\calF(X)$ and $g^*\calF(X)$ are equal to $\calF(X')$, so the lemma actually asserts that the empty blow ups are inserted at the same places.

(ii) The proof might look as casuistics, but it has a real meaning. Desingularizing $X$ and $X'$ simultaneously by $\calF(\oX)$ we have to compare their singularities and decide which one should be blown up earlier, and the trace of this information on $X'$ hides in the empty blow ups.
\end{rem}

\subsubsection{Restriction to affine subcategory}\label{affsec}

\begin{lem}\label{afflem}
Let $\gtC$ be any subcategory in $\QE$ such that

(i) $\gtC$ is closed under taking finite disjoint unions;

(ii) if $f\:Y\into X$ is an open immersion and $X\in\Ob(\gtC)$ then $f\in\Mor(\gtC)$ (in particular, $Y\in\Ob(\gtC)$).

Then any desingularization functor $\calF$ on $\gtC$ is uniquely determined by its restriction $\calF_\aff$ onto the full subcategory $\gtC_\aff$ formed by the affine schemes from $\gtC$. Moreover, any (strong) desingularization functor $\calF_\aff$ on $\gtC_\aff$ extends uniquely to a (strong) desingularization functor on $\gtC$.
\end{lem}

We only outline the proof of the lemma, since the argument is known. If an object $X$ of $\gtC$ is covered by open affine subschemes $X_1\. X_n$ then $X'=\coprod_{i=1}^n X_i$ is in $\gtC$ and we can consider the desingularization $\calF_\aff(X')\:\coprod_{i=1}^nY_i\to X'$. The blow up sequences $Y_i\longto X_i$ agree over the intersections $X_i\cap X_j$ by Lemma \ref{sublem}, and hence glue to a global desingularization $\calF(X)\:Y\to X$. We refer to \cite[3.37]{Kol} or \cite[\S7.1]{bmfun} for a detailed proof along this line.

\subsubsection{Abstract invariant}\label{invsec}

\begin{rem}\label{invrem}
(i) It is critical for Lemma \ref{afflem} to have disjoint unions in $\gtC$ since $\calF_\aff(X')$ contains more information than all blow up sequences $\calF_\aff(X_i)$. This is exactly the information about the order of the blow ups in $\calF_\aff(X')$ (in particular, sometimes we may simultaneously perform few blow ups from different
$\calF_\aff(X_i)$'s). Alternatively, one can notice that we glue trivial extensions of $\calF_\aff(X_i)$'s rather than these sequences themselves, and the list of inserted empty blow ups is the additional information. Obviously, we cannot combine a single blow up sequence $\calF_\aff(X')$ without this information.

(ii) If an algorithm is controlled by an invariant, as in \cite{Wl} or \cite{bmfun}, then the condition (i) in the lemma is redundant. In this case, the invariant dictates the order of the blow ups, so the gluing is obvious. Thus, we use disjoint unions to implicitly encode nearly the same information as contained in the
invariant.

(iii) To the best of my knowledge (which might be very incomplete), the idea to use disjoint unions instead of invariants is due to Kollar, see \cite[3.38]{Kol}. However, as we will immediately see both approaches are rigorously equivalent.

(iv) To make sense of the above claim we associate to any desingularization functor $\calF$ an ordered set as follows. If $X,Y$ are two schemes from $\gtC$ with desingularizations $\calF(X)\:X_n\longto X_0=X$ and $\calF(Y)\:Y_m\longto Y_0=Y$ and points $x\in X_i$ and $y\in Y_j$, then we say that $x$ and $y$ are $\calF$-equivalent if $\calF(X\coprod Y)$ simultaneously blows them up for the first time. We warn the reader that the equivalence class of $x$ depends on the whole tower $X_i\longto X$ rather then only on the local situation on $X_i$ (i.e. it depends on the history, at least to some extent). We denote the above equivalence class as $\inv(x)$ and the set of all equivalence classes as $\Inv(\calF)$. The latter set is naturally provided with a total order such that given points $x,y$ as above, $\calF$ first blows up the point with larger invariant. Theoretically, for each $x$ as above we have associated an invariant $\inv(x)\in\Inv(\calF)$ controlling $\calF$, though for practical applications one might often wish to have a more constructive description of the invariant.

(v) In principle, the structure of $\Inv(\calF)$ can be very (and unnecessary) complicated. For example, one can take an existing algorithm (e.g. of Bierstone-Milman) and appropriately refine the equivalence relation it defines on singularities. It is easy to see, that in this way one can construct $\Inv(\calF)$ that contains $\bfQ$ as an ordered subset. On the other hand, the invariant set of the algorithm of Bierstone-Milman (and probably, of many other known algorithms) is well-ordered and countable. It seems that its order can be bounded by $\omega^{\omega^2}$. The latter is a rough upper bound, and perhaps one can give a much sharper one. On the other hand it seems certain that one cannot hope to construct $\calF$ with a too small set of invariants, e.g. with $\Inv(\calF)=\bfN$. For example, a simple argument in \cite[6.3.1]{survey} shows  that $\Inv(\calF)$ is much larger than $\bfN$ for the algorithm of Bierstone-Milman and for the algorithm we will construct in this paper.
\end{rem}

\subsubsection{Functorial desingularization of varieties}
The strong desingularization from \cite{Hir} is neither functorial nor algorithmic. The desingularization from \cite{BM} is given in an explicit algorithmic way, and the functoriality of the latter algorithm with respect to smooth morphism was checked later in \cite{bmfun}. Finally, it was observed in \cite{bmt} that a general regular morphism between varieties can be reduced to smooth morphisms using certain limit procedures. Using this, compatibility with all regular morphisms was established in \cite{bmt}, and the result was also extended to finite disjoint unions of varieties. We formulate the latter theorem for reader's convenience, and it will be used essentially when we will generalize its assertion to $\QE_{p=0,\reg}$.

\begin{theor}\label{vardesth}
There exists functorial strong desingularization $\calF_\Var$ on the category $\Var_{p=0,\reg}$ such that the set $\Inv(\calF_\Var)$ of invariants is countable.
\end{theor}

\begin{rem}
It follows from \cite{bmt} that any desingularization functor $\calF_\Var$ on $\Var_{p=0,\reg}$ is induced from its restriction $\calF_\bfQ$ to the varieties over $\bfQ$ because any object of $\Var_{p=0,\reg}$ admits a regular morphism to a variety over $\bfQ$. It follows immediately that $\Inv(\calF_\Var)=\Inv(\calF_\bfQ)$, but the latter set is countable because the geometry over $\bfQ$ is countable (up to an isomorphism, there are countably many varieties, points and blow up sequences over $\bfQ$). So, countability in Theorem \ref{vardesth} is automatic.
\end{rem}

\subsection{Formal analogs}\label{formalsec}
In this section we recall very briefly basic notions from the theory of formal desingularization, and we refer to \cite[\S3]{temdes} for details. All formal schemes are assumed to be locally noetherian. Formal schemes and their ideals will be denoted as $\gtX$, $\gtY$, $\gtI\subset\calO_\gtX$, etc. Throughout \S\ref{formalsec} we assume that the characteristic is zero, i.e. all (formal) schemes are defined over $\bfQ$. This is only needed in order to be able to use Theorem \ref{NNth}. The reader that trusts Gabber's theorem \ref{gabth} (as the author does) may remove this assumption and replace each reference to Theorem \ref{NNth} with a reference to Theorem \ref{gabth}.

\subsubsection{Closed fiber}
The maximal ideal of definition defines a closed subscheme $\gtX_s$ called the closed fiber of $\gtX$. Topologically, $|\gtX|=|\gtX_s|$.

\subsubsection{Support of ideals}
We say that an ideal $\gtI\subset\calO_\gtX$ is {\em supported} on a closed formal subscheme $\gtZ=\Spf(\calO_\gtX/\gtJ)$ if $\gtJ^n\subset\gtI$ for large $n$. So, an ideal is open if and only if it is supported on $\gtX_s$.

\begin{rem}
(i) For an open ideal $\gtI$ one can also define its support set-theore\-tically as $|\Spf(\calO_\gtX/\gtI)|$ or as the associated reduced closed subscheme of $\gtX_s$ which is the reduction of $\Spec(\calO_\gtX/\gtI)$.

(ii) In general, one can define support set-theoretically using a generic fiber of $\gtX$ (there are different definitions of the latter in rigid, analytic or adic geometries).
\end{rem}


\subsubsection{Quasi-excellent formal schemes}\label{qeforsec}
We give the following definition of quasi-excellence, which is a priori more restrictive than its analog in \cite[\S3.1]{temdes}. A locally noetherian formal scheme $\gtX$ is {\em quasi-excellent} or {\em qe} if for any morphism $\Spf(A)\to\gtX$ of finite type the ring $A$ is qe. It follows from Theorem \ref{NNth}
that $\gtX$ is qe if and only if it admits a covering by open subschemes $\Spf(A_i)$ with qe rings $A_i$. In particular, Gabber's theorem implies that all reasonable definitions of quasi-excellence coincide (including this definition and the definition in \cite{temdes}) and that quasi-excellence is preserved under taking formal
completion along a closed subscheme.

\subsubsection{Formal blow ups}
In the affine case, a formal blow up $\hatBl_I(\Spf(A))\to\Spf(A)$ is the formal completion of the blow up $\Bl_I(\Spec(A))\to\Spec(A)$. This definition is compatible with formal localizations on the base and hence globalizes to the case of a general formal blow up $\hatBl_\gtI(\gtX)\to\gtX$ along an ideal $\gtI\subset\calO_\gtX$.

\subsubsection{Charts}\label{fchartsec}
Thus, the formal blow ups are glued from the {\em charts} $\Spf(A\{I/g\})$ with $g\in I$, where we set $A\{I/g\}=\wh{A[I/g]}$. We warn the reader that $A\{I/g\}$ does not have to be a subring of $A_{\{g\}}$ though $A[I/g]\subset A_g$. For example, set $R=k[[\pi]]$ and $A=R\{t\}$ so that $\gtX=\Spf(A)$ is a formal affine line over $\Spf(R)$. Consider the $\pi$-chart of the blow up along the open ideal $(t,\pi)$. It is of the form $\Spf(A\{t/\pi\})$, and one easily sees that $A\{t/\pi\}\toisom R\{t'\}$ where $t=\pi t'$. On the other hand, $A_{\{\pi\}}=0$ because $\pi$ is topologically nilpotent.

\subsubsection{Compatibility with usual blow ups}
Formal completion is compatible with (formal) blow ups, i.e. it takes blow ups of schemes to formal blow ups of formal schemes.

\subsubsection{Support}
If $\gtX$ is a formal $\gtS$-scheme and $\gtZ\into\gtS$ is a closed formal subscheme then we say that a formal blow up $\hatBl_\gtJ(\gtX)\to\gtX$ is {\em $\gtZ$-supported} if $\gtJ$ is $\gtZ\times_\gtS\gtX$-supported, i.e. $\gtI^n\calO_\gtX\subset\gtJ$ where $\gtZ=\Spf(\calO_\gtS/\gtI)$.

\subsubsection{Formal blow up sequences}
A {\em formal blow up sequence} $\gtf\:\gtX_n\longto\gtX$ is defined in an obvious way. Such a sequence is {\em $\gtZ$-supported} for a formal subscheme $\gtZ\into\gtX$ if all centers of $\gtf$ are $\gtZ$-supported.

\subsubsection{Singular locus}
Singular locus of a qe formal scheme is a reduced closed subscheme $\gtX_\sing$ or the corresponding ideal $\gtI\subset\calO_\gtX$. For an affine formal scheme $\Spf(A)$ this is the ideal that defines $\Spec(A)_\sing$, and this local definition globalizes because formal localizations are regular morphisms on qe formal schemes. Singular loci are compatible with formal completions: if a scheme $X$ is qe (and so its formal completion $\gtX=\hatX_\calI$ is qe by Gabber's theorem) then $\gtX_\sing$ is the completion of $X_\sing$ along $\calI\calO_{X_\sing}$ by \cite[3.1.4]{temdes}. One defines the {\em non-reduced locus} of a qe $\gtX$ similarly, and says that $\gtX$ is {\em regular} or {\em reduced} if the corresponding locus is empty.

\begin{rem}
Though I do not know such examples, it seems probable that regularity (and even reducedness) can be destroyed by formal localization of a noetherian adic ring. If this is the case then these notions do not make any sense for general noetherian formal schemes. At the very least, some examples show that reducedness can be destroyed by formal localizations in the non-noetherian case.
\end{rem}

\subsubsection{Desingularization}
A {\em (strong) desingularization} of $\gtX$ is defined similarly to the scheme case: it is an $\gtX_\sing$-supported formal blow up sequence with regular source (and regular centers).

\subsubsection{Rig-regularity}
We say that $\gtX$ is {\em rig-regular} if its singular locus is given by an open ideal (i.e. is $\gtX_s$-supported). Note that a desingularization in the rig-regular case is given by blowing up open ideals, i.e. blowing up formal subschemes which are usual schemes.

\begin{lem}\label{formseqlem}
Let $X$ be a qe scheme with a closed subscheme $Z$ and let $\gtX=\hatX_Z$ denote the formal completion, then

(i) any $Z$-supported blow up sequence $\gtf\:\gtX'\longto\gtX$ is the completion of a uniquely defined $Z$-supported blow up sequence $f\:X'\longto X$;

(ii) if $X_\sing\subset Z$ then any (strong) desingularization $\gtf\:\gtX'\longto\gtX$ is the completion of a uniquely defined (strong) desingularization $f\:X'\longto X$.
\end{lem}
\begin{proof}
Note that $\gtX$ is qe by Theorem \ref{NNth}. Any $Z$-supported closed subscheme of $\gtX$ is given by an open ideal, so it is defined already as a closed subscheme in $X$. So, the center $V_0$ of the first blow up of $\gtf$ is a closed subscheme of $X$ and we can set $X_1=\Bl_{V_0}(X)$. Then $\gtX_1$ is the completion of $X_1$, hence we can algebraize the center $V_1\into\gtX_1$ by $V_1\into X_1$, and proceed by induction on the length of $\gtf$. This proves (i), and (ii) follows from (i) and the compatibility of formal completions with singular loci.
\end{proof}

\subsubsection{Regular morphisms}\label{formregsec}
A morphism $\gtf\:\gtY\to\gtX$ between qe formal schemes is {\em regular} if there exist affine coverings $\gtX_i=\Spf(A_i)$ and $\gtY_i=\Spf(B_i)$ of $\gtX$ and $\gtY$ such that $\gtf(\gtY_i)\subset\gtX_i$ and the induced homomorphism $A_i\to B_i$ is regular.

\subsubsection{Completions of regular morphisms}
\begin{lem}\label{reglem0}
Let $A$ be a qe ring with an ideal $I$, $\hatA$ be its $I$-adic completion and $B$ be a noetherian $I$-adic $\hatA$-ring. Then the homomorphism $\hatA\to B$ is regular if and only if the homomorphism $A\to B$ is regular.
\end{lem}
\begin{proof} The direct implication is obvious since the completion homomorphism $A\to\hatA=C$ is regular by quasi-excellence of $A$. Conversely, suppose that $A\to B$ is regular. Since any prime ideal in $B$ is contained in an open prime ideal, in order to prove that $C\to B$ is regular we should show that for any open prime ideal $q\subset B$ with preimage $p\subset C$ the homomorphism $f\:C_p\to B_q$ is regular. Moreover, in view of Andre's theorem on localization of formal smoothness, see \cite{And}, it suffices to check that $f$ is formally smooth because $C$ is qe by Gabber's Theorem \ref{NNth}. Recall that by \cite[$\rm 0_{IV}$, 19.7.1 and 22.5.8]{ega} formal smoothness of the local homomorphism $f$ is equivalent to its flatness and geometric regularity of its closed fiber. The ideals $p$ and $r=p\cap A$ are open, hence $\hatA_r\toisom\hatC_p$. This isomorphism and faithful flatness of $C_p\to\hatC_p$ imply that $pC_p$ is the only prime ideal of $C_p$ over $r$. Hence $C_p/rC_p$ is local Artinian, and since it is regular by quasi-excellence of $A$, it has to be a field. In particular, we obtain that $rC_p=pC_p$. Thus, the closed fiber $B_q/pB_q\toisom B_q/rB_q$ of $f$ is geometrically regular over the residue field $C_p/pC_p\toisom A_r/rA_r$ by regularity of $A\to B$. Finally, the completion homomorphism $C_p\to\hatC_p$ is flat because $C_p$ is noetherian and the homomorphism $\hatC_p\toisom\hatA_r\to\hatB_q$ is flat because $A\to B$ is flat, hence $g\:C_p\to\hatB_q$ is flat and we deduce that $f$ is flat because $g$ is its composition with the faithfully flat completion homomorphism $B_q\to\hatB_q$.
\end{proof}

\begin{cor}\label{regcor}
Let $f\:Y\to X$ be a regular morphism between qe schemes, and let $\calI\subset\calO_X$ and $\calJ\supset\calI\calO_Y$ be ideals with the completions $\gtX=\hatX_\calI$ and $\gtY=\hatY_\calJ$. Then the completion $\gtf\:\gtY\to\gtX$ of $f$ is regular.
\end{cor}
\begin{proof}
We can assume that $X=\Spec(A)$ and $Y=\Spec(B)$, and let $\hatA$ and $\hatB$ be the $\calI$-adic and the $\calJ$-adic completions, respectively. Then the homomorphism $A\to B\to\hatB$ is regular, and therefore $\hatA\to\hatB$ is regular by Lemma \ref{reglem0}.
\end{proof}

\subsubsection{Regularity and affine subschemes}

\begin{lem}\label{reglem}
If a morphism $\gtf\:\gtY\to\gtX$ between qe formal schemes is regular and $\gtX'=\Spf(A')$ and $\gtY'=\Spf(B')$ are open formal subschemes of $\gtX$ and $\gtY$ such that $\gtf(\gtY')\subset\gtX'$, then the induced homomorphism $A'\to B'$ is regular.
\end{lem}
\begin{proof}
First we prove that regularity survives formal localizations. Namely, assume that $\gtY=\Spf(B)$, $\gtX=\Spf(A)$, the homomorphism $A\to B$ is regular, $A'=A_{\{f\}}$ and $B'=B_{\{g\}}$. Obviously, $A'\to B'$ is the completion of a regular homomorphism $A_f\to B_g$ which is a localization of $A\to B$. Hence $A'\to B'$ is a regular homomorphism by Corollary \ref{regcor}.

To complete the proof it now suffices to prove the following claim. Assume that $\gtY=\Spf(B)$, $\gtX=\Spf(A)$, $\gtX=\cup\gtX_i=\Spf(A_i)$ and $\gtY=\cup\gtY_i=\Spf(B_i)$ such that $A_i=A_{\{f_i\}}$, $B_i=B_{\{g_i\}}$, $\gtf(\gtY_i)\subset\gtX_i$ and the homomorphisms $A_i\to B_i$ are regular. In particular, the compositions $A\to A_i\to B_i$ are regular. Then we claim that the homomorphism $A\to B$ is regular. Assume to the contrary that $f\:\Spec(B)\to\Spec(A)$ is not regular. Since any point of $\Spec(B)$ specializes to a point of $\Spf(B)$, there exists a point $x\in\gtY\subset\Spec(B)$ such that the homomorphism $f$ is not regular at $x$. Since the morphisms $\Spec(B_i)\to\Spec(B)$ are regular, the composed morphisms $\Spec(B_i)\to\Spec(B)\to\Spec(A)$ are not regular at the preimage of $x$ by Lemma \ref{complem}(i). However, $x$ has a non-empty preimage in some $\Spec(B_i)$ because $\gtY_i$'s cover $\gtY$. This contradicts the regularity of $A\to B_i$, and therefore $f$ is regular.
\end{proof}

\section{Extending $\calF_\Var$ to schemes with small singular locus}\label{locsec}
Loosely speaking, the aim of \S\ref{locsec} is to extend the functor $\calF_\Var$ to generically reduced qe schemes over $\bfQ$ whose singular locus is sufficiently small. More precisely, we will extend $\calF_\Var$ to pairs $(X,Z)$ where $X$ is a generically reduced noetherian qe scheme over $\bfQ$ and $Z\into X$ is a Cartier divisor isomorphic to a disjoint union of varieties and containing $X_\sing$. This is an intermediate result towards our proof of the main Theorem \ref{desth}, so we do not pursue any generality in \S\ref{locsec}. The question of extending $\calF_\Var$ to wider classes of schemes was discussed in Remark \ref{difrem}. The construction of $\calF_\Var(X,Z)$ goes by completing $X$ along $Z$ and algebraizing the obtained formal variety, and the main difficulty is to prove that this construction is independent of the algebraization.


\subsection{Extending $\calF_\Var$ to formal varieties}\label{formsec}

\subsubsection{Formal varieties}
A noetherian formal scheme $\gtX$ is called a {\em formal variety} if its closed fiber $\gtX_s$ is a variety.

\begin{rem}
(i) Formal varieties are called special formal schemes in \cite[\S3.2]{temdes}.

(ii) It is easy to prove that an equicharacteristic $\gtX$ is a formal variety if and only if locally it is of the form $\Spf(k[T_1\. T_n][[S_1\. S_m]]/I)$ where $k$ is any field of definition of $\gtX_s$; see \cite[3.2.1]{temdes} for a proof. Note that the latter formal scheme is of finite type over $\Spf(k[[S_1\. S_m]])$ because of the isomorphism $$k[T_1\. T_n][[S_1\. S_m]]\toisom k[[S_1\. S_m]]\{T_1\. T_n\}$$
\end{rem}

\subsubsection{Rig-smoothness}\label{rigsmooth}
Let $A$ be an adic ring and let $B=A\{T_1\. T_n\}/I$ be topologically finitely generated over $A$. Following \cite{Elk}, one defines a Jacobian ideal $H\subseteq B$ that depends on some choices (see also \cite[\S3.3]{temdes}) and shows that its radical $H_{B/A}=\sqrt{H}$ depends only on $A$ and $B$. Moreover, the construction of $H_{B/A}$ is compatible with formal localizations, hence one obtains a reduced Jacobian ideal $H_{\gtY/\gtX}$ for any finite type morphism $f\:\gtY\to\gtX$. One can view the corresponding closed subscheme of $\gtY$ as the non-smoothness locus of $f$. In particular, $f$ is smooth if and only if $H_{\gtY/\gtX}=\calO_\gtY$. An arbitrary morphism $f\:\gtY\to\gtX$ is {\em rig-smooth} if it is of finite type and $H_{\gtY/\gtX}$ is open. The following remark will not be used, so the reader not familiar with non-archimedean geometry can safely skip it.

\begin{rem}
(i) Intuitively, rig-smoothness means that the "generic fiber" of $f$ is smooth.

(ii) If $\gtX=\Spf(k[[\pi]])$ then the generic fiber $f_\eta\:\gtY_\eta\to\gtX_\eta$ can be defined in the categories of rigid, analytic or adic spaces. Rig-smoothness of $f$ is equivalent to smoothness of $f_\eta^\rig$ by \cite[3.3.2]{temdes}, and by comparison of smoothness in rigid and adic categories, this also equivalent to smoothness of $f_\eta^\ad$. Smoothness for analytic spaces is more restrictive since one requires the boundary to be empty. Morphisms corresponding to smooth morphisms of rigid spaces are called quasi-smooth (or, sometimes, rig-smooth), thus $f$ is rig-smooth if and only if $f_\eta^\an$ is quasi-smooth.
\end{rem}

\subsubsection{Algebraization of formal varieties}
A formal variety is called {\em (locally) algebraizable} if (locally) it is isomorphic to a formal completion of a variety.

\begin{rem}
(i) It is well known that formal singularities can be non-algebraizable. So, a general formal variety does not have to be locally algebraizable.

(ii) The main algebraization tool is \cite[Th. 7]{Elk} by Ren\'ee Elkik. This theorem implies that if $A$ possesses a principal ideal of definition and $\Spf(B)\to\Spf(A)$ is a rig-smooth morphism then $\Spf(B)$ is
$A$-algebraizable.

(iii) It is an interesting question if the assumption about a principal ideal of definition can be weakened. Because of this assumption we have to introduce the class of principal formal varieties below.
\end{rem}

\subsubsection{Locally principal formal varieties}
By a {\em locally principal formal variety} we mean a pair $(\gtX,\gtI)$ consisting of a formal variety $\gtX$ and an invertible ideal of definition $\gtI\subset\calO_\gtX$, and we say that $(\gtX,\gtI)$ is {\em principal} if $\gtI$ is isomorphic to $\calO_\gtX$. By a (regular) morphism of locally principal formal varieties $f\:(\gtX',\gtI')\to(\gtX,\gtI)$ we mean a (regular) morphism $f\:\gtX'\to\gtX$ such that $\gtI'=\gtI\calO_{\gtX'}$.

\begin{rem}
(i) It is easy to see that an affine equicharacteristic formal scheme $\gtX$ with an ideal $\gtI$ is a principal formal variety if and only if it is of finite type over $(\gtS=\Spf(k[[\pi]]),(\pi))$, where $k$ is a field of definition of $\gtX_s$ and $\pi$ is a generator of the ideal of definition (we refer to \cite[3.2.3]{temdes} for details).

(ii) In the zero characteristic case, a principal formal variety $\gtX$ is rig-regular if and only if it is rig-smooth over an $\gtS$ as in (i).

(iii) It follows from (ii) and the theorem of Elkik that any affine rig-regular principal formal variety of characteristic zero is algebraizable; see \cite[3.3.1]{temdes} for details.
\end{rem}

\subsubsection{Desingularization of rig-regular locally principal formal varieties}\label{formalgsec}
Let $\hatVar_{p=0}$ denote the category of finite disjoint unions of rig-regular locally principal formal varieties of characteristic zero with regular morphisms.

\begin{theor}\label{formdesth}
There exists unique up to unique isomorphism desingularization functor $\hatcalF_\Var$ on $\hatVar_{p=0}$ such that $\hatcalF_\Var$ is compatible with $\calF_\Var$ under formal completions. Moreover, $\hatcalF_\Var$ is strong if and only if $\calF_\Var$ is strong.
\end{theor}

The theorem will be proved only in \S\ref{formdesthsec}. Since formal varieties from the theorem are locally algebraizable, the uniqueness is obvious. To prove existence we, at the very least, should show that if such an $\gtX$ admits two algebraizations, say $\gtX\toisom\hatX$ and $\gtX\toisom\hatY$, then the completions of the blow up sequences $\calF_\Var(X)$ and $\calF_\Var(Y)$ give rise to the same formal blow up sequence of $\gtX$. We will solve this problem in \S\ref{intrsec} by showing that the desingularization of $X$ and $Y$ is canonically determined already by a sufficiently thick infinitesimal neighborhood of $\gtX_s$. The main tool will be results of Elkik and Gabber-Ramero that we will recall in \S\ref{egrsec}.

\subsection{Scheme-theoretic singular locus}\label{egrsec}
Consider a morphism $f\:X=\Spec(B)\to S=\Spec(A)$ and let $f_\sing\subset X$ be its singular locus. It is a natural question if one can meaningfully extend this set-theoretic notion to a scheme-theoretic one.

\subsubsection{Jacobian ideals}\label{jacobsec}
A rough version of such a notion was introduced in \cite{Elk}. Namely, one considers a closed immersion $X\into W=\bfA^N_S$, chooses generators of the ideal $J\subseteq F=A[t_1\. t_N]$ defining $X$ in $W$, and assigns to this datum a Jacobian ideal $H_F$ on $W$ whose cosupport is $f_\sing$ (it is denoted $H_B$ in \cite[\S0.2]{Elk}). In particular, this ideal may be used to define singular locus of morphisms of formal schemes (see \S\ref{rigsmooth}), and it plays a crucial role in the proof of the famous Popescu's theorem, see \cite{Po}. The disadvantages of this definition is that it depends on the choices and $V(H_F)$ does not have to be a subscheme of $X$. This indicates that $V(H_F)$ gives an upper bound on a scheme-theoretic singular locus, and indeed there is a more precise definition by Gabber-Ramero.

\subsubsection{Gabber-Ramero ideal}\label{grideal}
Gabber-Ramero introduced in \cite[5.4.1]{GR} an ideal $$\bfH_F=\Ann_F\Ext^1_B(\bfL_{B/A},J/J^2)$$ where $\bfL_{B/A}$ is the cotangent complex of Illusie. By \cite[5.4.2(iii)]{GR} this ideal annihilates $\Ext^1_B(\bfL_{B/A},N)$ for any $B$-module $N$. Since $J\subseteq\bfH_F$, we can view $V(\bfH_F)$ as a closed subscheme of $X$ given by the ideal $\bfH_{B/A}=\bfH_FB$. The latter is the biggest ideal of $B$ that annihilates any module of the form $\Ext^1_B(\bfL_{B/A},N)$, hence it depends only on the homomorphism $A\to B$. By \cite[5.4.2]{GR}, the vanishing locus of $\bfH_{B/A}$ is the singular locus of $f$, in particular, $B$ is smooth over $A$ if and only if $\bfH_{B/A}=B$. Thus, it is natural now to set $f_\sing=V(\bfH_{B/A})$, as a subscheme of $X$. Moreover, it is shown in \cite[5.4.6]{GR} that $V(\bfH_F)$ is a subscheme of $V(H_F)$ (regardless to the choices in the definition of $H_F$), supporting the intuition that $V(H_F)$ is an upper bound on $f_\sing$. We will use Gabber-Ramero ideals in the sequel, though we will indicate in comments how one could use only the results of \cite{Elk} instead. First, let us list basic compatibility properties of $\bfH_{B/A}$.

\begin{prop}\label{grprop}
Let $f:A\to B$ be a ring homomorphism.

(i) Gabber-Ramero ideal can only increase under base changes. Namely, if $A\to A'$ is a homomorphism and $B'=B\otimes_AA'$ then $\bfH_{B/A}B'\subseteq\bfH_{B'/A'}$.

(ii) Gabber-Ramero ideals satisfy the following transitive domination property: if $g\:B\to C$ is another homomorphism then $\bfH_{B/A}\bfH_{C/B}\subseteq\bfH_{C/A}$. In particular, if $g$ is smooth (resp. $f$ is smooth) then $\bfH_{B/A}C\subseteq\bfH_{C/A}$ (resp. $\bfH_{C/B}\subseteq\bfH_{C/A}$).
\end{prop}
\begin{proof}
Claim (i) is precisely \cite[5.4.2(i)(a)]{GR}. To prove (ii) we recall that by \cite[II.2.1.2]{Ill} one associates to $A\to B\to C$ an exact transitivity triangle $$\bfL_{C/B}[-1]\to \bfL_{B/A}\otimes_B^\bfL C\to \bfL_{C/A}\to \bfL_{C/B}$$ In particular, for any $C$-module $N$ we obtain an exact sequence $$\Ext^1_C(\bfL_{C/B},N)\to\Ext^1_C(\bfL_{C/A},N)\to\Ext^1_C(\bfL_{B/A}\otimes_B^\bfL C,N)$$
It suffices to prove that $\bfH_{B/A}\bfH_{C/B}$ annihilates the middle module in this sequence. Since $\bfH_{C/B}$ annihilates the left one, we should only check that $\bfH_{B/A}$ annihilates the module on the right. But this follows from the isomorphism $$\Ext^1_C(\bfL_{B/A}\otimes_B^\bfL C,N)\toisom\Ext^1_B(\bfL_{B/A},\RHom_C(C,N))\toisom\Ext^1_B(\bfL_{B/A},N)$$
where the first map is an adjunction isomorphism.
\end{proof}

\subsubsection{Principal affine pairs}
By a {\em principal affine pair} over a field $k$ we mean a pair $(X,\calI)$ consisting of an affine $k$-variety $X=\Spec(A)$ with a principal invertible ideal $\calI\subset\calO_X$ corresponding to $I\subset A$. Thus, $I=(\pi)$ for a non-zero divisor $\pi$. Sometimes, we will denote such pair as $(A,I)$. The support of $\calI$ will be called the {\em closed fiber} and we will denote it $X_s$. The closed fiber underlies schemes $X_n=\Spec(A/I^n)$, and the pair $(X_n,\calI_n)$, where $\calI_n=\calI\calO_{X_n}$, will be called the {\em $n$-th fiber} of $(X,\calI)$. The henselization and the formal completion of the pair will be denoted $(X^h,\calI^h)$ and $(\hatX,\hatcalI)$, respectively. We have natural regular morphisms $\hatX\to X^h\to X$ and closed immersions of $X_n$ into the above schemes which are compatible with the ideals and induce isomorphisms on the $m$-th fibers for $m\le n$.

\subsubsection{Morphisms of pairs}
A morphism between principal pairs $f\:(X',\calI')\to (X,\calI)$ is a morphism $h\:X\to X'$ such that $\calI'=\calI\calO_{X'}$. If $X$, $X'$ and $h$ are defined over $k$ then we say that $f$ is a {\em $k$-morphism}. Also, we define a morphism $\of\:(\oX',\ocalI')\to(\oX,\ocalI)$ between $n$-th fibers, henselizations or completions as a morphism $\oh\:\oX\to\oX'$ that respects the ideals. Note that we do not impose any condition on the original pairs in this definition. A morphism $\of$ as above is said to be {\em regular, isomorphism}, etc. if $\oh$ is so. Finally, to any morphism $f$ we obviously associate $n$-th fibers, henselization and completion which will be denoted $f_n$, $f^h$ and $\hatf$, and similarly to any morphism $\of$ between completions, henselizations or $n$-th fibers we can associate a morphism $\of_m$ between the $m$-th fibers, where $m\le n$.

\subsubsection{Conductor}\label{condsec}
Assume now that $(X,\calI)$ is a principal pair and $f\:Y=\Spec(B)\to X$ a finite type morphism. Then we define the {\em conductor} of $f$ to be the minimal number $r$ such that $I^r\subseteq\bfH_{B/A}$. If no such number exists then the conductor is infinite. Note that the conductor is finite if and only if $f$ is smooth over the complement of $V(I)$.

\begin{rem}\label{condrem}
(i) It follows from Proposition \ref{grprop}(i) that the conductor does not increase under base changes.

(ii) It follows from Proposition \ref{grprop}(ii) that the conductor does not increase when one replaces $B$ with a smooth $B$-algebra $C$.

(iii) So far, we did not use that $I$ is principal and invertible. Nevertheless, it is not so clear if the notion of conductor makes too much sense in general, so we prefer to impose the restrictions on $I$ from the beginning.
\end{rem}

\subsubsection{The lifting theorem}
The following lifting result will be a critical tool in the sequel. In some sense it indicates that, indeed, the "non-smoothness of $f$ is bounded by $t^r$".

\begin{theor}[Gabber-Ramero]\label{elkth}
Assume that $f:(Y,\calJ)\to(X,\calI)$ is a morphism of principal pairs of conductor $r$. If the pair $(X,\calI)$ is henselian, $n>r$ a number, and $\ve_n\:X_n\to Y_n$ is a section of $f_n$ that can be lifted to a section $\ve'_{n+r}$ of $f_{n+r}$, then $\ve_n$ can also be lifted to a section $\ve\:X\to Y$ of $f$.
\end{theor}
\begin{proof}
This is, actually, the claim of \cite[5.4.13]{GR} when one takes trivial $I$ in the loc.cit.
\end{proof}

\begin{rem}
(i) Note that $\ve$ does not have to agree with $\ve'_{n+r}$ on the $m$-th fibers for $m>n$.

(ii) The above theorem of Gabber-Ramero improves an analogous result of Elkik (see \cite[Th. 2 and Cor. 1]{Elk}), in which a weaker conductor is defined through Jacobian ideals. Note that the bounds in Elkik's theorem are implicit because $\calI$ can be any finitely generated ideal in \cite{Elk}. Tracing the proof, one may check that when $\calI$ is principal and invertible the bound is essentially as we stated (in terms of Jacobian ideals), e.g. see \cite[Lemma 1]{Elk} and the proof of \cite[Th. 1]{Elk}.
\end{rem}

\subsection{Intrinsic dependence of desingularization on the singular locus}\label{intrsec}
Intuitively, it is natural to expect that for a scheme $X$ all information about its singularities, including the desingularization information, is contained in a "sufficiently thick" closed subscheme $Y$ with $|Y|=X_\sing$. It is a very interesting and difficult question how to define such property of $Y$ rigorously, but we will not study it here in general. We will only consider a very special case of what we call Elkik pair, which suffices to prove Theorem \ref{formdesth}.

\subsubsection{Elkik pairs}\label{elksec}
Until the end of \S\ref{intrsec} we assume that the characteristic is zero. By {\em Elkik pair over $k$} we mean a principal affine pair $(X,\calI)$ over $k$ such that $X_\sing$ is contained in the closed fiber $X_s$. Note that for any generator $\pi$ of $I$ the generic fiber $X_\eta$ of the induced morphism $f\:X\to\bfA^1_k$ is a regular scheme, and hence $f$ is smooth along $X_\eta$ by the characteristic zero assumption. In particular, removing from $X$ few bad fibers we can achieve that $f$ is smooth outside of $X_s$ and then the conductor of $f\:(X,\calI)\to(\bfA^1_k,(\pi))$ is a finite number that we denote $r(X,\pi)$.

\subsubsection{Conductor of Elkik pairs}
Instead of studying the question whether $r(X,\pi)$ depends only on $(X,\calI)$, we define the {\em conductor} $r(X,\calI)$ as the minimal possible value of $r(X,\pi)$, where $\pi$ runs over the set of all generators of $\calI$. Any $\pi$ for which the minimum is achieved will be called a {\em featured} generator.

\subsubsection{Recovery of a henselian Elkik pair from a sufficiently thick fiber}
It is a well known, but difficult to find in the literature, folklore that the results of \cite{Elk} imply that the henselization of an Elkik pair is determined up to an isomorphism by a sufficiently thick fiber. Let us show that, indeed, this is a simple consequence of the lifting theorem. We work with Gabber-Ramero version, but obviously one could use only \cite{Elk}, obtaining slightly worse bounds. In the following proposition we say that a morphism is {\em henselian-smooth} if it is a henselization of a smooth morphism.

\begin{prop}\label{fibprop}
Assume that $(X,\calI)$ and $(X',\calI')$ are Elkik pairs over $k$, $r$ is the conductor of $(X,\calI)$ and $n>r$ is a number. Then for any smooth $k$-morphism $f_n\:(X'_n,\calI'_n)\to(X_n,\calI_n)$ that can be lifted to a smooth morphism $$f'_{n+r}\:(X'_{n+r},\calI'_{n+r})\to(X_{n+r},\calI_{n+r})$$ there also exists a lifting of $f_n$ to a henselian-smooth morphism of henselizations $f\:(X'^h,\calI'^h)\to(X^h,\calI^h)$. Moreover, if $f_n$ (or even $f_0$) is an isomorphism then $f$ is an isomorphism, and so the henselizations are $k$-isomorphic if and only if the $m$-th fibers are $k$-isomorphic for a single $m>2r$. In particular, the fiber $X_{2r+1}$ determines $X^h$ up to an isomorphism.
\end{prop}
\begin{proof}
Fix any featured generator $\pi$ of $\calI$ and denote its images in $\calI^h$ and $\calI_n$ also by $\pi$. Then $X$ and all its derived schemes are provided with compatible morphisms to $S=\Spec(k[\pi])$. Let $\pi'\in \calI'_n$ be the image of $\pi$ under the homomorphism induced by $f_n$, and denote by $\pi'$ any lifting of $\pi'$ to $\calI'$. We view $X'^h$ as an $S$-scheme via the morphism taking $\pi'$ to $\pi$. The morphism $X\times_SX'^h\to X'^h$ is a base change of $X\to S$, hence its conductor is bounded by $r$ by Remark \ref{condrem}(i). By Theorem \ref{elkth} there exists a section $X'^h\to X\times_SX'^h$ which lifts $$(f_n,\Id_{X'_n})\:X'_n\to X_n\times_S X'_n$$ Projecting this section onto $X$ gives a morphism $\phi\:X'^h\to X$ that lifts $f_n$. By the universal property of henselizations, the latter morphism factors through a morphism $f\:X'^h\to X^h$ that lifts $f_n$.

Note that $X'^h=\projlim X'_\alp$ where $X'_\alp\to X'$ are \'etale morphisms inducing isomorphisms on henselizations. By \cite[$\rm IV_4$, 8.13.1]{ega}, $\phi$ factors through a morphism $g\:Y=X'_\alp\to X$ for large enough $\alp$, and clearly $f=g^h$. We claim that the fibers of $g$ are smooth. Indeed, $g_n=\phi_n=f_n$ is smooth, hence $g_0$ is smooth and it suffices to check that $g_m$ is flat for any $m>n$. Let $X=\Spec(A)$ and $Y=\Spec(B)$. By local criterion of flatness \cite[Th. 49, (4)$\implies$(1)]{Mat} with respect to the ideal $\pi A$ it suffices to show that $\pi^mA/\pi^{m+1}A\otimes_{A_0}B_0\toisom\pi^mB/\pi^{m+1}B$. This follows from the fact that, since $A$ and $B$ have no $\pi$-torsion, $\pi^mA/\pi^{m+1}A\toisom A/\pi A=A_0$ and $\pi^mB/\pi^{m+1}B\toisom B_0$. Smoothness of the fibers implies (actually is equivalent to) formal smoothness of the completion homomorphism $\hatA\to\hatB$, and the latter implies that $g$ is smooth along $X_s$. In particular, replacing $Y$ with a small enough neighborhood of the closed fiber we achieve that $g$ is smooth. In particular, $f=g^h$ is henselian-smooth. Finally, if $f_0=g_0$ is an isomorphism then $g$ is strictly \'etale along $X_s$ and hence its henselization $f$ is an isomorphism.
\end{proof}

In the above proof we also showed that the henselian-smooth morphism $f$ is approximated by a smooth morphism $g$ between Elkik pairs. We record this result too for the sake of latter referencing.

\begin{cor}\label{fibcor}
Let $(X,\calI)$, $(X',\calI')$, $r$, and $n$ be as in Proposition \ref{fibprop}. Then there exist a smooth morphism $g\:Y\to X$ and an \'etale morphism $\phi\:Y\to X'$ such that $\phi\times_XX_s$ is an isomorphism (equivalently, $\phi^h\:Y^h\toisom X'^h$ is an isomorphism) and the induced morphism $g_n\circ\phi_n^{-1}\:X'_n\to X_n$ equals to $f_n$.
\end{cor}

\subsubsection{Naive restriction of a desingularization on an Elkik fiber}\label{naiverestrsec}
All schemes and morphisms in \S\ref{naiverestrsec} are defined over a field $k$. Let $(X,\calI)$ be an Elkik pair over $k$, and assume that $f\:X^{(p)}\longto X^{(0)}$ is a (strong) desingularization of $X=X^{(0)}$. In particular all centers $V^{(i)}\into X^{(i)}$ of $f$ sit over $X_s$, and hence we can choose $l$ such that each $V^{(i)}$ is contained in the $l$-th fiber $X^{(i)}_l$. To simplify the notation we set $X'=X^{(1)}$, $V'=V^{(1)}$ and $V=V^{(0)}$. Set $r=\max_{0\le i\le p}r(X^{(i)},\calI^{(i)})$. For any number we define the {\em naive restriction} of $f$ onto the $n$-th fiber as the sequence $f_n\:X^{(p)}_n\longto X_n^{(0)}$ provided with the closed subschemes $V^{(i)}\into X_n^{(i)}$. A serious disadvantage of the naive restriction is that the situation is not fully controlled by $X_n$ for the following two reasons. First, although $X$ is determined by $X_n$ up to an isomorphism whenever $n>2r$, not any automorphism of $X_n$ has to lift to $X$. However, any automorphism of $X_{n-r}$ does lift, so if $V\subseteq X_{n-r}$ then at least $X'_{n-r}$ is determined up to an $X_{n-r}$-automorphism. Second, there might be non-trivial $X_{n-r}$-automorphisms $\sigma$ of $X'_{n-r}$ even though each $X$-automorphism of $X'$ is trivial. So, the subscheme $V\into X$ completely determines the morphism $X'\to X$, but given only $X_n$ and $V\into X_{n-r}$ we cannot reconstruct the morphism $X'_{n-r}\to X_{n-r}$. Fortunately, it turns out that any such $\sigma$ is the identity modulo large enough power of $\calI$, so sufficiently thick fibers $X'_m$ are determined uniquely.

\begin{lem}\label{automlem}
Keep the above notation, and assume that $n>\max(r,l)+r$. Then the $X_n$-scheme $X'_{n-l-r}$ is determined by $X_n$ and $V$ up to a unique $X_n$-isomorphism.
\end{lem}
\begin{proof}
Let us explain first how one can non-canonically construct $X'_{n-l-r}$ from $X_n$ and $V$. Find any realization of $(X_n,\calI_n)$ as the $n$-th fiber of an Elkik pair $(Y,\calI')$. The original $(X,\calI)$ is a possible choice, but since we are proving intrinsic dependency on $X_n$ there are other equally good alternatives. Set $Y'=\Bl_V(Y)$ and consider the fiber $Y'_{n-l-r}$. To prove the lemma we should show that the above a priori non-canonical construction is actually canonical, and we will do that by establishing a canonical $X_n$-isomorphism $f'_{n-l-r}\:Y'_{n-l-r}\toisom X'_{n-l-r}$. By Proposition \ref{fibprop} there exists an isomorphism $f^h\:Y^h\toisom X^h$ which induces identity on the $(n-r)$-th fibers. In the sequel it will be more convenient to work with the formal completion $\hatf\:\hatY\toisom\hatX$ and the completed blow up $\hatY'\to\hatY$ (the reason for switching to formal schemes is that the theory of henselian schemes, their blow ups, etc., was not developed in the paper). Since $V\into X_{n-r}$, $\hatf$ induces an isomorphism $\hatf'\:\hatY'\toisom\hatX'$ whose $(n-l-r)$-th fiber is an isomorphism $f'_{n-l-r}$. It remains to prove uniqueness of $f'_{n-l-r}$, and this follows from the following claim by taking $m=n-r$.

{\em Assume that $\phi$ is an automorphism of $\hatY$ that preserves $V$ and induces identity on $Y_m$. Then the induced automorphism $\phi'$ of $\hatY'$ induces identity on $Y_{m-l}$.} The latter claim reduces to a simple computation on charts of $\hatY'$. Let $\hatY=\Spf(A)$ and $\hatV=V(J)$ for an ideal $J\subset A$ containing $\pi^l$, where $\pi$ is a generator of $\calI$. By \S\ref{fchartsec}, $\hatY'$ is covered by charts $Z_g=\Spf(A\{J/g\})$ with $g\in J$. Note that $\phi'$ moves $Z_g$ to $Z_{\phi(g)}$ and takes a function $f/g\in A\{J/g\}$ to $\phi(f)/\phi(g)\in A\{J/\phi(g)\}$. In the intersection $Z_{g\phi(g)}=\Spf(A\{J^2/g\phi(g)\})$ we have that
$$\frac{f}{g}-\phi\left(\frac fg\right)=\frac{f}{g}-\frac {f+\pi^ma}{g+\pi^mb}=\pi^{m-l}\cdot\frac{\pi^l}{g}\cdot\frac{bf-ag}{\phi(g)}\in\pi^{m-l}A\{J^2/g\phi(g)\}$$
This proves that $\phi'_{m-l}$ acts trivially on $(Z_{g\phi(g)})_{m-l}$. Also, the same argument shows that the open immersions $(Z_{g\phi(g)})_{m-l}\into (Z_g)_{m-l}$ and $(Z_{g\phi(g)})_{m-l}\into (Z_{\phi(g)})_{m-l}$ are actually isomorphisms. So, $\phi'_{m-l}$ preserves the charts and acts trivially on them, and hence is an identity.
\end{proof}

\subsubsection{Restriction of a desingularization on an Elkik fiber}\label{restrsec}
The lemma implies that for large enough $n$ the subscheme $V^{(0)}\into X^{(0)}_{n^2}$ defines $X^{(1)}_{n(n-1)}$ up to a unique isomorphism, the subscheme $V^{(1)}\into X^{(1)}_{n(n-1)}$ defines $X^{(2)}_{n(n-2)}$ up to a unique isomorphism, etc. The sequence of morphisms $X^{(p)}_{n(n-p)}\to\dots\to X^{(0)}_{n^2}$ induced by the desingularization $f\:X^{(p)}\longto X^{(0)}$ and the centers $V^{(i)}\into X^{(i)}_{n(n-i)}$ for $0\le i\le p-1$ will be called the {\em $n$-th restriction of $f$} onto the Elkik fiber and will be denoted $f|_{X_{n^2}}$. The restriction allows to encode the desingularization in Elkik fibers in a way which is canonical up to the choice of $n$. The latter is not a real trouble, since any choice of large enough $n$ does the job equally well, and the subschemes $V^{(i)}\into X^{(i)}_{n(n-i)}$ and $V^{(i)}\into X^{(i)}_{N(N-i)}$ with $N\ge n$ are identified by the closed immersion $X^{(i)}_{n(n-i)}\into X^{(i)}_{N(N-i)}$. For this reason, we will not worry about $n$ in the sequel, and each time $n$ appears in the notation of a restriction we will assume that it is large enough.

\subsubsection{Ghost blow ups}
In order to formalize the above data, we define a {\em ghost blow up} as a morphism $f\:X'\to X$ and a closed subscheme $V\into X$ such that $f$ is an isomorphism over the complement of $V$. Naturally, $V$ is called the center of $f$ and the ghost blow up is {\em empty} if its center is. Obviously, $f|_{X_{n^2}}$ is a ghost blow up sequence which is also contracted, i.e. does not contain empty ghost blow ups.

\subsubsection{Compatibility with smooth $k$-morphisms}
Let $(X,\calI)$ and $(Y,\calI')$ be Elkik pairs, let $f\:X'\longto X$ and $g\:Y'\longto Y$ be some desingularizations, and consider the restrictions $\of=f|_{X_{n^2}}$ and $\og=g|_{Y_{n^2}}$ for large enough $n$ (depending on $X$ and $Y$). Given a flat morphism $h\:(Y_{n^2},\calI'_{n^2})\to(X_{n^2},\calI_{n^2})$ between the Elkik fibers we denote by $h^*(\of)$ the base change of $\of$ with respect to $h$, i.e. the tower $h^*(X^{(p)}_{n(n-p)})\to\dots\to Y^{(0)}_{n^2}$ with the subschemes $h^*(V^{(i)})$ where by definition $h^*(Z)=Z\times_{X_{n^2}}Y_{n^2}$ for an $X_{n^2}$-scheme $Z$. We say that $\of$ and $\og$ are {\em compatible with respect to $h$} if $\og$ is obtained from $h^*(\of)$ by eliminating empty ghost blow ups.

\begin{lem}\label{smoothlem}
Let $\calF$ be a desingularization algorithm for $k$-varieties which is compatible with smooth $k$-morphisms. Then the restriction of $\calF$ on Elkik fibers is compatible with smooth $k$-morphisms between the fibers in the following sense: let $(X,\calI)$ and $(Y,\calI')$ be Elkik pairs over $k$, and let $r$ be the conductor of $(X,\calI)$. If $n$ is large enough (depending on $X$ and $Y$) then for any smooth $k$-morphism $h\:(Y_{n^2+r},\calI'_{n^2+r})\to(X_{n^2+r},\calI_{n^2+r})$ the $n$-th restrictions $\calF(X)|_{X_{n^2}}$ and $\calF(Y)|_{Y_{n^2}}$ are compatible with $h|_{n^2}$.
\end{lem}
\begin{proof}
Corollary \ref{fibcor} tells that $h_{n^2}$ lifts to a smooth morphism of Elkik pairs $f\:(Z,\calI'')\to(X,\calI)$ and an \'etale morphism of Elkik pairs $g\:(Z,\calI'')\to(Y,\calI')$ such that $g^h$ is an isomorphism and $h_{n^2}=f_{n^2}\circ g^{-1}_{n^2}$. By our assumption, $\calF$ is compatible with $f$ and $g$, so $f^*(\calF(X))$ is a trivial extension of $\calF(Z)$ and similarly for $g$. Moreover, $g^*(\calF(Y))$ actually coincides with $\calF(Z)$ because $Y_\sing\subseteq Y_s$ and hence $Y_\sing$ is contained in the image of $Z$. In particular, $h_{n^2}^*(\calF(X)|_{n^2})$ is a trivial extension of $\calF(Z)|_{n^2}\toisom\calF(Y)|_{n^2}$.
\end{proof}

\subsubsection{Compatibility with regular morphisms}
Now, we are going to generalize Lemma \ref{smoothlem} to the case of regular morphisms $h\:{Y_{n^2}}\to X_{n^2}$. Note that even if $h$ is smooth and $X,Y$ are $k$-varieties, $h$ does not have to be a $k$-morphism. For this reason we do not have an analog of Corollary \ref{fibcor} anymore. Our strategy will be to reduce to $\bfQ$-varieties using approximation and to use that Lemma \ref{smoothlem} covers all regular morphisms between $\bfQ$-varieties.

\begin{prop}\label{Elkprop}
Let $\calF$ be a desingularization functor for varieties of characteristic zero that is compatible with all regular morphisms. Then the restriction of $\calF$ on Elkik fibers is compatible with regular morphisms between the fibers in the following sense: let $(X,\calI)$ and $(Y,\calI')$ be Elkik pairs over fields $k$ and $l$, respectively. Then there exists a number $r$ depending on $(X,\calI)$ such that for any large enough $n$ depending on $X$ and $Y$ and a regular morphism $h\:(Y_{n^2+r},\calI'_{n^2+r})\to(X_{n^2+r},\calI_{n^2+r})$, the $n$-th restrictions $\calF(X)|_{X_{n^2}}$ and $\calF(Y)|_{Y_{n^2}}$ are compatible with $h|_{n^2}$.
\end{prop}
\begin{proof}
By \cite[7.5]{bmt} $X$ is the filtered projective limit of a family $X_{\alp}$ of affine $\bfQ$-varieties with smooth transition morphisms, and then the projections $X\to X_{\alp}$ are regular. Moreover, taking only $\alp\ge\alp_0$ for a large enough $\alp_0$ we can easily achieve that $\calI$ is defined already on each $X_\alp$, and each pair $(X_{\alp},\calI_{\alp})$ is an Elkik fiber over $\bfQ$. Set $r=r(X_{\alp_0},\calI_{\alp_0})$; then the conductor of each pair $(X_{\alp},\calI_{\alp})$ does not exceed $r$ by Remark \ref{condrem}(ii). Define in a similar way a family $Y_{\alp'}$ with limit $Y$. Since $\calF$ is compatible with regular morphisms, $\calF(X)$ is induced from $\calF(X_\alp)$, and hence the restriction of $\calF$ to $X_{n^2}$ (resp. $Y_{n^2}$) is induced from its restriction to each $X_{\alp,n^2}$ (resp. $Y_{\alp',n^2}$). Fix any $\alp'$, then by \cite[$\rm IV_4$, 8.13.1]{ega} the composed morphism $h_{\alp'}\:X_{n^2+r}\to Y_{n^2+r}\to Y_{\alp',n^2+r}$ is induced from a morphism $h_{\alp,\alp'}\:X_{\alp,n^2+r}\to Y_{\alp',n^2+r}$ with large enough $\alp$. The latter morphism does not have to be smooth but its restriction onto a neighborhood $U$ of the image of $X_{n^2+r}$ is smooth by Lemma \ref{complem} because the composition of $h_{\alp,\alp'}$ with the regular morphism $X_{n^2+r}\to X_{\alp,n^2+r}$ is the regular morphism $h_{\alp'}$. Replacing $X_{\alp,n^2+r}$ with $U$ we can assume that $h_{\alp,\alp'}$ is a smooth morphism of $\bfQ$-varieties. Then $h_{\alp,\alp'}|_{n^2}$ is compatible with the restriction of $\calF$ by Lemma \ref{smoothlem}, and therefore $h|_{n^2}$ is compatible with the restriction as well.
\end{proof}

The following remark will not be used so we leave the proof to the interested reader.

\begin{rem}
One can also prove that if $(X,\calI)$ and $(X_\alp,\calI_\alp)$ are as in the proof of Proposition \ref{Elkprop} then $r(X,\calI)=\min_\alp r(X_\alp,\calI_\alp)$. It then follows that one can take $r=r(X,\calI)$ in the formulation of the proposition.
\end{rem}

\subsection{Functorial desingularization via algebraization}

\subsubsection{Proof of Theorem \ref{formdesth}}\label{formdesthsec}
Let $(\gtX,\gtI)$ be an affine principal rig-regular formal variety of characteristic zero. Choose a generator $\pi\in\gtI$, then a rig-smooth morphism $\gtX\to\gtS=\Spf(k[[\pi]])$ arises and therefore there exists an algebraization $(\gtX,\gtI)\toisom(\hatX,\hatcalI)$, where $(X,\calI)$ with $\calI=(\pi)$ is an Elkik pair over $k$. The completion of the desingularization $\calF_\Var(X)$ of $X$ is a desingularization $\gtf\:\gtX'\to\gtX$, and it follows from Proposition \ref{Elkprop} that $\gtf$ intrinsically depends on $(\gtX,\gtI)$. Indeed, if the latter is isomorphic to another completed Elkik pair $(\gtY,\gtJ)$ then the isomorphism induces isomorphism of all fibers and already the restrictions onto large enough fiber (depending on $\gtX$) are isomorphic by \ref{Elkprop}. Thus, we can denote $\gtf$ as $\hatcalF_\Var(\gtX)$. Furthermore, Proposition \ref{Elkprop} implies that the functor $\hatcalF_\Var$ (which is defined, so far, in the principal affine case) is compatible with any regular morphism $h\:(\gtY,\gtI')\to(\gtX,\gtI)$ because the fibers $h_n\:(\gtY_n,\calI'_n)\to(\gtX_n,\calI_n)$ are regular. It follows by the standard gluing argument from \S\S\ref{affsec}--\ref{invsec} that the definition of $\hatcalF_\Var$ extends to all rig-regular locally principal formal varieties of characteristic zero and their disjoint unions, and the obtained desingularization is compatible with all regular morphisms.

\subsubsection{Desingularization of schemes with small singular locus}\label{smallsec}
Consider a category $\gtC_\small$ as follows. Objects of $\gtC_\small$ are pairs $(X,Z)$, where $X$ is a generically reduced noetherian qe scheme of characteristic zero and $Z$ is a Cartier divisor in $X$ which contains $X_\sing$ and is a disjoint union of varieties. Morphisms in $\gtC_\small$ are morphisms $(X',Z')\to(X,Z)$ such that $X'\to X$ is regular and $Z'=X'\times_XZ$.

\begin{theor}\label{locth}
The functor $\calF_\Var$ extends to a desingularization functor $\calF_\gtC$ which assigns to a pair $(X,Z)$ a desingularization of $X$ in a way functorial with all morphisms from $\gtC_\small$. If $\calF_\Var$ is strong then $\calF_\gtC$ is strong.
\end{theor}
\begin{proof}
To define $\calF_\gtC$ we note that $\gtX:=\hatX_Z$ is a rig-regular locally principal formal variety by \cite[3.1.5(ii)]{temdes}, hence it admits a desingularization $\gtf=\hatcalF_\Var(\gtX)$ by Theorem \ref{formdesth}. Since $\gtX$ is rig-regular, the centers of its desingularization are $\gtX_s$-supported. By Lemma \ref{formseqlem} $\gtf$ algebraizes to a desingularization $f\:X'\longto X$, and $f$ is strong if and only if $\gtf$ is strong. It remains to observe that given a morphism $f\:(X',Z')\to(X,Z)$ we obtain a morphism of their completions $\gtf\:\gtX'\to\gtX$ because $Z'=X'\times_XZ$. Moreover, if $f$ is regular then $\gtf$ is regular by Corollary \ref{regcor}. Since $\hatcalF_\Var$ is compatible with any regular morphism $\gtf$, we obtain that $\calF_\gtC$ is functorial with respect to all morphisms from $\gtC_\small$.
\end{proof}

\begin{rem}\label{locrem}
(i) Though it is reasonable to expect that $\calF_\gtC$ is functorial with respect to all regular morphisms $X'\to X$, we did not prove that. Moreover, it is even unclear if the desingularization of $(X,Z)$ is actually independent of the choice of $Z$ (in case, there are few possibilities).

(ii) Since it cannot cause any confusion, we will denote $\calF_\gtC$ as $\calF_\Var$ in the sequel.
\end{rem}

\section{Construction of $\calF$}\label{fsec}

\subsection{Induction on codimension}\label{indsec}

\subsubsection{Unresolved locus}
When working on (strong) desingularization of a scheme $X$ we will use the following terminology: given an $X_\sing$-supported blow up sequence $f\:X'\longto X$ by the {\em unresolved locus} of $f$ we mean the smallest closed set $f_\unr=T\subset X$ such that $f$ induces a (strong) desingularization on $X\setminus T$. Note that $T=f(X'_\sing)$ (resp. $T$ is the union of $f(X'_\sing)$ and the images of all singular loci of the centers of $f$).

\subsubsection{Desingularization up to codimension $d$}
We say that a blow up sequence $f\:X'\longto X$ is a {\em (strong) desingularization up to codimension $d$} if $f_\unr\subset X^{>d}$. This happens if and only if $f$ induces a (strong) desingularization on a neighborhood on $X^{\le d}$.

\begin{rem}\label{desrem}
(i) If $d=\dim(X)$ then desingularization is the same as desingularization up to codimension $d$. On the other side, if $X'\longto X$ is a desingularization up to codimension $d-1$ then $X'$ can have singularities in any codimension and their local structure can be even worse than that of the original $X$. However, the global structure of $X'_\sing$ is relatively simple because it is contractible to a finite set of closed points of $X$. In particular, $X'_\sing$ is a disjoint union of proper varieties, while $X_\sing$ can be any qe scheme of dimension smaller than $d$.

(ii) The desingularization of integral schemes in \cite{temdes} is constructed as a blow up sequence $X_n\longto X_0=X$, where each $X_d$ desingularizes $X$ up to codimension $d$ and the blow up $X_{d+1}\to X_d$ is $T$-supported for some closed $T\subset X$ disjoint from $X^{\le d}$. Thus, the desingularization is built successively by improving the situation over $X^1$, $X^2$, etc. To construct the blow up $X_{i+1}\to X_i$ we should care only for the preimages of finitely many "bad" points from $X^{i+1}$, and, similarly to (i), this reduces the problem to the case when $X_\sing$ is a variety. The latter case is reduced to desingularization of varieties by completing and algebraizing, similarly to our desingularization of $\gtC_\small$ in \S\ref{smallsec}.

(iii) We will adopt a similar strategy here with two modifications as follows: (a) we will also insert new $T$-supported blow ups in the middle of the sequence in order to correct the old centers over $X^{d+1}$ (in case of strong desingularization) and in order to make the preimage of each bad point to a Cartier divisor, and (b)
we will work with functors $\calF^{\le d}$ on $\QE_{p=0,\reg}$ rather than with desingularizations of single schemes. In a sense, we will construct $\calF$ by establishing an exhausting filtration $\calF^{\le d}$ by its blow up "subsequences" that desingularize each $X$ up to codimension $d$. This plan will be precisely formulated in \S\ref{brsec}.
\end{rem}

\subsubsection{Equicodimensional blow up sequences}
We say that a blow up sequence or a desingularization $f\:X_n\longto X_0$ of a locally noetherian scheme $X=X_0$ is {\em equicodimensional} if for each center $V_i$ of $f$ there exists a number $d$ such that $V_i$ is disjoint from the preimage of $X^{<d}$ and $V_i$ is $X^{\le d}$-admissible. Let $g_i\:X_i\to X$ denote the natural projection, then the above condition can be re-stated as follows: $g_i(V_i)$ is of pure codimension $d$ and for the discrete set $T_i=g_i(V_i)\cap X^d$ of its maximal points the set $g_i^{-1}(T_i)\cap V_i$ is schematically dense in $V_i$. In the above situation we say that $V_i$ is {\em of pure $X$-codimension $d$}.

\subsubsection{Filtration by codimension}
\begin{lem}\label{filtlem}
Let $f\:X_n\longto X_0=X$ be an equicodimensional blow up sequence and let $U$ be obtained from $X$ by removing the images of all centers of $f$ of $X$-codimension $\ge d$. Then there exists a unique blow up sequence $f^{\le d-1}\:X_m\longto X$ such that all centers of $f^{\le d-1}$ are of $X$-codimension strictly smaller than $d$ and
$f^{\le d-1}\times_XU$ is obtained from $f\times_XU$ by removing the empty blow ups.
\end{lem}
\begin{proof}
The lemma dictates what the restriction of $f^{\le d-1}$ over $U$ is. Since the centers of $f^{\le d-1}$ are of $X$-codimension strictly smaller than $d$, the whole $f^{\le d-1}$ must be the pushforward of $f^{\le d-1}|_U$ with respect to the open immersion $U\into X$ (i.e. we simply take the centers of $f^{\le d-1}$ to be the schematical closures of the centers of $f^{\le d-1}|_U$, as was observed in \S\ref{supsec}(a)).
\end{proof}

\subsubsection{The strategy of constructing $\calF$ and $\calF^{\le d}$}\label{brsec}
Let $\calF$ be an equicodimensional desingularization functor. Then we can define functors $\calF^{\le d}$ by setting $\calF^{\le d}(X)=(\calF(X))^{\le d}$. The procedure of removing centers of large $X$-codimension is compatible with regular morphisms, so $\calF^{\le d}$ is a functor in the same sense as $\calF$ is: if $f\:Y\to X$ is regular then $\calF^{\le d}(Y)$ is a trivial extension of $\calF^{\le d}(X)\times_XY$. Moreover, each $\calF^{\le d}$ is an up to
codimension $d$ desingularization functor because $\calF(X)|_U=\calF^{\le d}(X)|_U$ for a neighborhood $U$ of $X^{\le d}$. The sequence of functors $\calF^{\le d}$ is compatible in the sense that $(\calF^{\le d})^{\le e}=\calF^{\le e}$ for any pair $e\le d$.

\begin{lem}\label{filt2lem}
(i) Each equicodimensional desingularization functor $\calF$ defines a compatible sequence $\{\calF^{\le d}\}_{d\in\bfN}$ of equicodimensional desingularizations up to codimension $d$.

(ii) Conversely, assume that $\{\calF^{\le d}\}_{d\in\bfN}$ is a compatible sequence of equicodimensional desingularizations up to codimension $d$, and for any scheme $X$ from $\QE_{p=0,\reg}$ and numbers $e\ge d$ the centers of $\calF^{\le e}(X)$ of $X$-codimension $>d$ are $T_d$-supported, where $T_d=\calF^{\le d}_\sing(X)$ is the unresolved locus of $\calF^{\le d}(X)$. Then the sequence $\{\calF^{\le d}(X)\}_{d\in\bfN}$ stabilizes for large $d$'s, and hence the sequence $\{\calF^{\le d}\}_{d\in\bfN}$ gives rise to an equicodimensional desingularization functor $\calF$ on $\QE_{p=0,\reg}$.
\end{lem}
\begin{proof}
The assertion of (i) was observed earlier. Assume given a sequence $\calF^{\le d}$ and a scheme $X$ as in (ii). The condition on the centers implies that $\calF^{\le d}(X)|_{X\setminus T_d}=\calF^{\le e}(X)|_{X\setminus T_d}$ and hence $T_e\subseteq T_d$. Thus, $T_1\supseteq T_2\supseteq T_3\dots$ and hence this sequence stabilizes by noetherian induction. However, $T_n$ is of codimension $n$, hence the only possibility for stabilization is that $T_n=\emptyset$ for large enough $n$, and then the sequence $\calF^{\le d}(X)$ stabilizes for $d\ge n$.
\end{proof}

We will construct successively a compatible family $\calF^{\le d}$ as in Lemma \ref{filt2lem}(ii). Each $\calF^{\le d+1}$ will be obtained by inserting few blow ups into $\calF^{\le d}$ in order to correct it over the codimension $d$ points of $\calF^{\le d}_\sing$. First we have to prove few easy claims about inserting blow ups into a blow up sequence, and this will be done in \S\ref{opersec}.

\subsection{Operations with blow up sequences}\label{opersec}

\subsubsection{Pushing forward with respect to closed immersions}\label{pushsec2}
\begin{lem}\label{pushlem2}
Let $X_0$ be a scheme with a closed subscheme $V_0=V'_0$ and let $g\:V_n\longto V_0$ be a blow up sequence of length $n$. Then there exists a unique blow up sequence $f\:X_n\longto X_0$ of length $n$ such that for each $1\le i<n$ the $i$-th strict transform $V'_i\into X_i$ of $V_0$ is $V_0$-isomorphic to $V_i$ and for each $0\le i<n$ the $i$-th center of $f$ is contained in $V'_i$ and is mapped isomorphically onto the $i$-th center of $g$ by the $V_0$-isomorphism $V'_i\toisom V_i$.
\end{lem}
\begin{proof}
The center of the first blow up of $f$ must coincide with that of $g$. Then the strict transform of $V_0$ is $V_0$-isomorphic to $V_1$ (and the isomorphism is unique because both are modifications of $V_0$). The isomorphism $V'_1\toisom V_1$ dictates the choice of the second center of $f$, etc.
\end{proof}

In the situation of the lemma, we will say that $f$ is the {\em pushforward} of $g$ with respect to the closed immersion $V_0\into X_0$.

\subsubsection{Extending}
\begin{defin}\label{extdef}
Assume that $f\:X_n\longto X_0$ is a blow up sequence of length $n$ and with centers $V_i$, $0\le m\le n$ is a number and $U\into X_0$ is an open subscheme such that $V_i$ is $U$-admissible for each $m<i<n$. Let, furthermore, $g\:X'_m\longto X_m$ be any $(X_0\setminus U)$-supported blow up sequence of length $n'$, then by an {\em extension} of $f$ with $g$ we mean a blow up sequence of the form $$f'\:X'_n\stackrel{f'_{n-1}\circ\dots\circ f'_m}\longto X'_m\stackrel{g}\longto X_m\stackrel{f_{m-1}\circ\dots\circ f_0} \longto X_0$$ of length $n+n'$ obtained from $f$ by inserting $g$ before $f_m$ and such that the following conditions are satisfied: (a) after the base change with respect to the open immersion $U\into X_0$, $f'$ becomes a trivial extension of $f$; (b) for each $i\ge m$ the center $V'_i$ of $f'_i\:X'_{i+1}\to X'_i$ is $U$-admissible. By {\em successive extending} of $f$ we mean applying the above extension operation few times.
\end{defin}

\begin{lem}\label{extlem}
Given $f\:X_n\longto X_0$ and $g\:X'_m\longto X_m$ as in definition \ref{extdef}, there exists a unique extension $f'$ of $f$ by $g$. The center $V'_m$ of $f'_m$ is naturally isomorphic to the strict transform of the center $V_m$ of $f_m$ under $g$, i.e. $V'_m=g^!(V_m)$.
\end{lem}
\begin{proof}
The conditions (a) and (b) leave no choice in the definition of the centers $V'_i$ for $m\le i<n$: we must have that $V_i\times_XU\toisom V'_i\times_XU$ and $V'_i$ is the schematic closure of this scheme in $X'_i$. These rules dictate an inductive construction of the canonical sequence of blow ups $f'_{i+1}\:X'_{i+1}=\Bl_{V'_i}(X'_i)\to X'_i$ starting with $f'_m$. Then it is clear that the sequence satisfies all properties an extension should satisfy. Regarding the second claim, we just use the definition of the strict transform and the fact that both $V'_m$ and $V_m$ have a common schematically dense open subscheme which is the
preimage of $U$.
\end{proof}

\subsubsection{Merging}

\begin{lem}\label{merglem}
Assume that $X$ is a scheme with pairwise disjoint closed subschemes $T_1\. T_n$, $U_i=X\setminus(\coprod_{j\neq i}T_j)$, $U=\coprod_{i=1}^n U_i$, and $g\:U'\longto U$ is a $(\coprod_{i=1}^n T_i)$-supported blow up sequence. Then there exists a unique $(\coprod_{i=1}^n T_i)$-supported blow up sequence $f\:X'\longto X$ such that $f|_{U_i}=g|_{U_i}$ for each $1\le i\le n$.
\end{lem}

In the situation described in this obvious lemma we will say that $f$ is {\em merged} from the the blow up sequence $g$ or from its components $g\times_U U_i$.

\subsubsection{Compatibility with flat base changes}
Since blow up sequences are compatible with flat base changes in the sense of \S\ref{flatsec}, one can check straightforwardly that all constructions from \S\ref{opersec} are also compatible. So, we obtain the following lemma.

\begin{lem}\label{compatlem}
The operations of pushing forward, merging and extending blow up sequences are compatible with flat base changes.
\end{lem}

\subsection{The main theorem}

Now we have all necessary tools to construct a strong desingularization functor $\calF$ on $\QE_{p=0,\reg}$ from the functor $\calF_\Var$ which was extended to a functor on $\gtC_\small$ by Theorem \ref{locth}. Note, however, that the functor $\calF$ will not coincide with $\calF_\Var$ even on varieties, since we must build the new desingularization functor from scratch for the sake of compatibility.

\begin{proof}[Proof of Theorem \ref{desth}]
By Lemma \ref{filt2lem} it suffices to build a compatible sequence of functors $\calF^{\le d}$ which provide an equicodimensional (resp. strong) desingularization up to codimension $d$ and such that the centers of $\calF^{\le d}$ of $X$-codimension $d$ sit over $T_{d-1}=\calF^{\le d-1}_\sing$. The construction will be done inductively, and we start with empty $\calF^{\le 0}$ since generically reduced schemes are regular in codimension $0$. Thus, we can assume that the sequence $\calF^{\le 0}\.\calF^{\le d-1}$ is already constructed, and our aim is to construct $\calF^{\le d}$. First we will construct $\calF^{\le d}(X)$ for a single scheme $X$, and then we will check that the construction is functorial. The required sequence will be obtained by extending the blow up sequence $f=\calF^{\le d-1}(X)\:X_m\longto X_0=X$ few times. To simplify notation, after each extension we will denote the obtained blow up sequence as $f\:X_m\longto X_0$, but this should not cause any confusion. By our assumption, $T_{d-1}$ is a closed subset of $X^{\ge d}$, hence it has finitely many points of codimension $d$, which are the generic points of the irreducible components of $T_{d-1}$ of codimension $d$. Let $T$ denote the set of these points.

{\it Extension 0.} We denote by $\oT$ the Zariski closure of $T$ with the reduced scheme structure and extend $f$ in the sense of Definition \ref{extdef} by inserting $g\:\Bl_\oT(X)\to X$ as the first blow up. As an output we obtain a blow up sequence $\calF^{\le d}_0(X)\:X'_m\longto X'_0\to X_0=X$ of length $m+1$ where the first center (the inserted one) is regular over a neighborhood of $T$. As agreed above, we set $f=\calF^{\le d}_0(X)$ and increase $m$ by one after this step. We claim that the scheme-theoretic preimage of $\oT$ in each $X_i$ with $i>0$ is a Cartier divisor $D_i$. Indeed, this is obvious for $D_1$ and for larger $i$'s we use the following lemma.

\begin{lem}
If $X'\to X$ is a blow up and $D\into X$ is a Cartier divisor then the locally principal closed subscheme $D'=D\times_XX'$ of $X'$ is a Cartier divisor.
\end{lem}
\begin{proof}
We can work locally on $X$, so let $X=\Spec(A)$, $D=(f)$, and $X'$ be the blow up along an ideal $I\subset A$. It is suffices to check the claim for a chart $\Spec(A[\frac{I}{a}])$ of the blow up, where $a\in I$. Since $f$ is a regular element of $A$, it is a regular element of $A_a$ and hence also a regular element of $A[\frac{I}{a}]\subseteq A_a$. Thus, $D'$ is a Cartier divisor in $X'$.
\end{proof}

{\it Extensions $1\. n$.} This step is needed only in the case of strong desingularization (otherwise one can run it as well, but this would only unnecessarily complicate the algorithm). The last $n$ centers of $\calF^{\le d}_0(X)$ are regular over $X^{\le d}\setminus\oT$ but do not have to be so over $T$. We remedy this problems by $n$ successive extensions. Let us describe the $i$-th one. It obtains as an input a blow up sequence $f=\calF^{\le d}_{i-1}(X)$ in which only the last $n-i$ centers can be non-regular over $T$ and outputs a blow up sequence $\calF^{\le d}_i(X)$ with only $n-i-1$ bad blow ups in the end. By our assumption, $X_{i+1}\to X_i$ is the first blow up of $f$ whose center $W$ can be non-regular over $T$. The intersection of $W_\sing$ with the preimage of $T$ can be non-empty, but it is definitely contained in $W\cap X_i^{\le d}\subseteq W^{\le d-1}$. Consider the blow up sequence $\calF^{\le d-1}(W)\:W'\longto W$, which exists by the induction assumption. Its centers $V_j$ can have singularities only over $W^{\ge d}$, in particular, the image of $(V_j)_\sing$ in $X$ is contained in $X^{>d}$. By Lemma \ref{pushlem2}, the pushforward $X'_i\longto X_i$ of $\calF^{\le d-1}(W)$ with respect to the closed immersion $W\into X_i$ is a blow up sequence with centers $V_j$. In particular, its centers are regular over $X^{\le d}$.

Let now $f'\:X'_m\longto X'_i\longto X_i\longto X_0$ be obtained from $f$ by inserting $X'_i\longto X_i$ instead of $X_{i+1}\to X_i$ as in Definition \ref{extdef}. By Lemma \ref{extlem} the center of $X'_{i+1}\to X'_i$ is the strict transform of $W$, hence it is $W'$. Since $W'$ is regular over $X^{\le d}$ by the construction, only the last $i-1$ blow ups of $f'$ can be non-regular over $T$ (the blow ups from the sequence $X'_m\longto X'_{i+1}$). So, we can set $\calF^{\le d}_i(X)=f'$.

\begin{rem}\label{redambrem}
The scheme $W$ does not have to be integral. So, we essentially exploited here that the resolution functor $\calF$ is defined for all reduced schemes.
\end{rem}

{\it Extension $n+1$.} At this stage we already have a blow up sequence $f=\calF^{\le d}_n(X)$ such that all its centers are regular over $X^{\le d}$. The only problem is that though the singular locus of $X_m$ is disjoint from the preimage of $X^{\le d}\setminus T$ it can intersect the preimage of $T$. For any $x\in T$ consider the pro-open subscheme $X_x:=X_m\times_X\Spec(\calO_{X,x})$ of $X_m$ and the scheme-theoretic preimage $D_x=(D_m)|_{X_x}$ of $x$. Clearly, $(X_x)_\sing\subseteq D_x$, and $D_x$ is a Cartier divisor by Extension 0. Thus, the pair $(X_x,D_x)$ is an object of $\gtC_\small$, and hence so is $(\coprod_{x\in T} X_x,\coprod_{x\in T} D_x)$. Applying $\calF_\Var$ to the latter pair we obtain a list of (resp. strong) desingularizations $f_x\:X'_x\longto X_x$ for all $x\in T$. (Actually, each $f_x$ is $\calF_\Var(X_x,D_x)$ saturated with few synchronizing empty blow ups.) Let $U\into X$ be an open neighborhood of $X^{\le d}$ such that the closures $\ox\in U$ of distinct points $x\in T$ are pairwise disjoint, and set $U_m=X_m\times_XU$. Define $g_x\:U_x\longto U_m$ as the pushforward of $f_x$ with respect to the pro-open immersion $X_x\into U_m$. Since each $g_x$ is $\ox$-supported, Lemma \ref{merglem} implies that we can merge all $g_x$'s into a single blow up sequence $g\:U'_m\longto U_m$. Finally, we define $f'\:X'_m\longto X_m$ to be the pushforward of $g$ with respect to the open immersion $U_m\into X_m$, and set $\calF^{\le d}(X)=f\circ f'$. Then $\calF^{\le d}(X)$ coincides with $\calF^{\le d-1}(X)$ over $X^{\le d}\setminus T$ and coincides with $f\circ f_x$ over each $\Spec(\calO_{X,x})$ for $x\in T$. By our construction the latter is a (resp. strong) desingularization of $\Spec(\calO_{X,x})$, hence the constructed $\calF^{\le d}(X)$ is a (resp. strong) desingularization of $X$ up to codimension $d$.

\begin{rem}
To illustrate a certain flexibility (and non-canonicity) of the method, note that we could use the blow up sequences $\calF_\Var(X_x,D_x)$ instead of $f_x$'s in the construction of $\calF^{\le d}(X)$ (in other words, we could omit all empty blow ups in $f_x$'s). Although such choice would work as well, our construction seems to be more natural.
\end{rem}

It remains to check that $\calF^{\le d}$ is functorial. By Lemma \ref{compatlem} the operations of pushing forward, extending and merging of blow ups are compatible with a regular (and even flat) base change $h\:\tilX\to X$, and we claim that it follows easily that all intermediate constructions in our proof are functorial. Indeed, the set of the maximal points of $h^{-1}(T)$ is exactly the set $\tilT$ of points of $\tilX^{\le d}$ over which $\calF^{\le d-1}(\tilX)$ is not a desingularization. Hence the Zariski closure of $\tilT$ coincides with $\oT\times_X\tilX$ in Extension 0, and therefore the blow up sequence $\calF^{\le d}_0$ is functorial. In the $i$-th Extension we have that $\tilW=W\times_X\tilX$ by inductive functoriality of $\calF^{\le d}_{i-1}$. Hence $\calF^{\le d-1}(\tilW)$ is a trivial extension of $\calF^{\le d-1}(W)\times_X\tilX$ by functoriality of $\calF^{\le d-1}$, and so $\calF_i^{\le d}$ is functorial. The last Extension is dealt with similarly.
\end{proof}

\begin{rem}
The same proof actually applies to a more general situation, which might include, for example, schemes of arbitrary characteristic. Let $\gtC$ be a subcategory of $\QE_\reg$ which is closed under blow ups and taking subschemes. Define a category $\gtC_\small$ as usually, i.e. the objects are pairs $(X,Z)$ with $X$ in $\gtC$ and $Z$ a Cartier divisor containing $X_\sing$ and isomorphic to a finite disjoint union of varieties, and the morphisms are the regular ones. Then the same argument as above shows that starting with a (strong) desingularization functor $\calF_\Var$ on $\gtC_\small$ one can construct a (strong) desingularization functor $\calF$ on $\gtC$. Obviously, in the case of non-strong desingularization one can skip extensions $1\. n$.
\end{rem}

\section{Desingularization in other categories}\label{catsec}

\subsection{Desingularization of stacks}

\subsubsection{Stacks}
Let $\gtX$ be an Artin stack, in particular, $\gtX$ admits a smooth covering $p\:U\to\gtX$ by a scheme. Given such a covering set $R=U\times_\gtX U$ with the projections $s$ and $t$ onto $U$, let $m$ denote the projection $p_{13}$ of $R\times_{s,U,t}R\toisom U\times_\gtX U\times_\gtX U$ onto $R$, and let $e\:U\to R$ denote the diagonal. Then $(U,R,s,t,m,e)$ is a scheme groupoid called a {\em smooth atlas} of $\gtX$. We will consider only smooth atlases and say that a stack is {\em qe} if it admits a qe smooth atlas, i.e. an atlas in which $U$ is qe (and so $R$ is also qe).

\subsubsection{Blow ups of stacks}
Any 1-morphism of stacks $\gtY\to\gtX$ lifts to a morphism of atlases if one chooses a fine enough atlas of $\gtY$. We say that a morphism $f\:\gtY\to\gtX$ is a {\em blow up} along a closed substack $\gtZ$ if it admits an atlas $(f_1,f_0)\:(Y_1\rra Y_0)\to(X_1\rra X_0)$ with $f_0$ the blow up along $Z_0=\gtZ\times_\gtX X_0$ and $f_1$ the blow up along $Z_1=\gtZ\times_\gtX X_1$. Note that in this case one can choose $X_1\rra X_0$ to be any atlas of $\gtX$ and then take $Y_i=\gtY\times_\gtX X_i$.

\subsubsection{Desingularization}
\begin{theor}\label{stackdes}
The blow up sequence functor $\calF$ extends uniquely to the $2$-category of generically reduced noetherian qe stacks over $\bfQ$.
\end{theor}
\begin{proof}
Given a generically reduced qe stack $\gtX$, find a qe atlas $s,t\:R\rra U$ of $\gtX$. Note that $U$ and $R$ are automatically generically reduced. By functoriality of $\calF$ the smooth morphisms $s$ and $t$ extend to smooth morphisms $t_n\. t_0=t$ and $s_n\. s_0=s$ between the whole blow up sequences $\calF(R)\:R_n\longto R$ and
$\calF(U)\:U_n\longto U$. Moreover, the groupoid composition map $m\:R\times_{s,U,t}R\to R$ is smooth, hence it extends to a tower of groupoid maps $m_i\:R_i\times_{s_i,U_i,t_i}R_i\to R_i$. Thus, we have actually constructed a tower of groupoid blow ups $(R_n\rra U_n)\longto(R\rra U)$, which gives rise to a blow up sequence $\calF(\gtX)\:\gtX_n\longto\gtX_0=\gtX$. Since $U_n$ is regular the stack $\gtX_n$ is regular.

Now, we have to check that the construction is independent of the chart. This reduces to applying the functoriality of $\calF$ few more times in the situation when two charts are dominated by a third one and comparing the corresponding blow up sequences. Finally, the compatibility of the construction with regular $1$-morphisms and $2$-isomorphisms between them is checked similarly, so we skip the details.
\end{proof}

\subsection{Desingularization of formal schemes and analytic spaces}

\subsubsection{Categories}
Let $\gtC'$ be any of the following categories: noetherian qe formal schemes over $\bfQ$, quasi-compact complex analytic spaces (maybe non-separated), quasi-compact $k$-analytic spaces of Berkovich or quasi-compact rigid $k$-analytic spaces for a complete non-Archimedean field $k$ of characteristic zero, which is non-trivially valued in the rigid case. We will be interested in the full subcategory $\gtC$ of $\gtC'$ whose objects have nowhere dense non-reduced locus.

\subsubsection{Regularity}
There is a natural notion of regular morphisms in all these categories: see \S\ref{formregsec} for the formal case, regularity is smoothness in the complex analytic and rigid analytic cases, and it is quasi-smoothness as defined by Ducros in \cite{duc2} in the case of Berkovich analytic spaces (for strictly analytic spaces this is equivalent to rig-smoothness, and in general this means that the morphism becomes rig-smooth after large enough ground field extension).

\subsubsection{Excellence}
Recall that Berkovich analytic spaces are excellent by \cite{duc}. As for complex analytic spaces $X$, let us say that a Stein compact $V\subset X$ is {\em excellent} if so is the ring $\calO_X(V)$. We claim that $X$ can be covered by excellent Stein compacts. Indeed, there exists an open covering $X=\cup_{i\in I}V_i$ such that each $V_i$ admits a closed immersion into an open polydisc $M_i$. Choose closed polydiscs $B_i\subset M_i$ such that $X_i=B_i\cap V_i$ cover $X$. The rings $\calO_{M_i}(B_i)$ are excellent by \cite[Th. 102]{Mat}, hence so are their quotients $\calO_X(X_i)$, i.e. $X_i$'s are excellent Stein compacts.

\begin{rem}
In general, $\calO_X(V)$ does not have to be noetherian, though it is always noetherian for a semi-algebraic $V$ (cf. \cite{Fri}). It seems plausible that if $\calO_X(V)$ is noetherian then it is qe, but this result seems to be missing in the literature.
\end{rem}

\subsubsection{Desingularization}
Let $\gtC_\reg$ be obtained from $\gtC$ by removing all non-regular morphisms.

\begin{theor}\label{formth}
Let $\gtC_\reg$ be as above. Then $\calF$ induces a strong desingularization on $\gtC_\reg$ by formal/analytic blow up sequence functor $\calF_\gtC$. For analytic and rigid analytic spaces, it is also compatible with ground field extensions.
\end{theor}
\begin{proof}
Each compact space $X$ from $\gtC$ can be covered by finitely many charts $X_i$ which are affine formal schemes, affinoid spaces, or excellent Stein compacts. Similarly, we cover each intersection $X_i\cap X_j$ by spaces $X_{ijk}$ which are affine formal schemes, affinoid spaces, or excellent Stein compacts, though these coverings may be infinite (if $X$ is not quasi-separated). The rings of functions $A_i=\calO_X(X_i)$ and $A_{ijk}=\calO_X(X_{ijk})$ are excellent noetherian rings in each of these cases, and it is known that the homomorphisms $\phi_{ijk}\:A_i\to A_{ijk}$ are regular. Also, the non-reduced locus on $X_i$ is compatible with the non-reduced locus on $\calX_i:=\Spec(A_i)$ under the morphism $X_i\to\calX_i$ of locally ringed spaces, and so $\calX_i$ is generically reduced. The desingularization blow up sequence $\calF(\calX_i)$ induces a formal/analytic blow up sequence $\calF_\gtC(X_i)\:X'_i\longto X_i$ by completing/analytifying the centers. These sequences agree on the intersections because $\calF$ is compatible with the regular homomorphisms $\phi_{ijk}$. So, $\calF(X_i)$'s glue to a single blow up sequence $\calF_\gtC(X')\:X'\longto X$ which desingularizes $X$. Compatibility of $\calF_\gtC$ with regular morphisms follows from compatibility of the original $\calF$.
\end{proof}

\subsection{Desingularization in the non-compact setting}\label{noncomsec}

\subsubsection{Categories}
Let $\ogtC'$ be any of the following categories: locally noetherian qe schemes over $\bfQ$, locally noetherian qe stacks over $\bfQ$, locally noetherian qe formal schemes over $\bfQ$, complex analytic spaces, Berkovich $k$-analytic spaces or rigid spaces over $k$ for a complete non-Archimedean field $k$ of characteristic zero. As earlier, let $\ogtC$ be the full subcategory of $\ogtC'$ whose objects have nowhere dense non-reduced locus and let $\ogtC_\reg$ be obtained from $\ogtC$ by removing all non-regular morphisms.

\subsubsection{Blow up hypersequences}
By a {\em hypersequence} we mean a totally ordered set $I$ with an initial element $0$ and such that any element $i\in I$ possesses a successor which will be denoted $i+1$. By a {\em hypersequence $\{X_i\}_{i\in I}$ in $\ogtC$} we mean a set of objects of $\ogtC$ ordered by a hypersequence $I$ and provided with a transitive family of morphisms $f_{ji}\:X_j\to X_i$ for each $j\ge i$ (so $f_{ii}=\Id_{X_i}$). If $I$ is a hypersequence with a finite subsequence $i_0<i_1<\dots<i_n$, and $X_{i_n}\to X_{i_{n-1}}\to\dots\to X_{i_0}$ is a blow up sequence, then we can define its {\em trivial extension} $\{X_i\}_{i\in I}$ by taking $f_{kj}$ to be an empty blow up for any any pair $j,k\in I$ such that either $i_m<j\le k\le i_{m+1}$ for some $0\le m<n$, or $j\le k\le i_0$,  or $i_n\le j\le k$. By a {\em blow up hypersequence} we mean a hypersequence $\{X_i\}_{i\in I}$ such that each morphism $f_i:=f_{i+1,i}$ is a blow up, and each point $x\in X_0$ possesses a neighborhood $U$ for which the hypersequence $\{X_i\times_XU\}_{i\in I}$ is a trivial extension of its finite subsequence.

\begin{rem}
(i) Any blow up hypersequence converges to an object $X_\infty$ (if $I$ has a maximal element $i$ then $X_\infty=X_i$). So, if $I$ has no maximal element we can extend $\{X_i\}_{i\in I}$ to a blow up hypersequence $\{X_i\}_{i\in I\cup\{\infty\}}$.

(ii) Clearly, an infinite blow up sequence which stabilizes over compact subobjects of $X_0$ is a blow up hypersequence. Often one can perform few blow ups simultaneously thus "shrinking" the initial hypersequence to a hypersequence ordered by $\bfN\cup\infty$, i.e. to a usual sequence $\dots\to X_1\to X_0$ augmented by its limit. For example, this is obviously the case when $X_0$ is a disjoint union of compact components. However, we saw in Remark \ref{invrem} that functorial properties are destroyed by such operation, since one cannot ignore the order of the blow ups even when $X_0$ is a disjoint union of compact pieces. The existing desingularization algorithms can be extended to non-compact case functorially only when one allows blow up hypersequences because the resolving invariant takes values in complicated ordered sets rather than in $\bfN$ (see Remark \ref{invrem}(v)).
\end{rem}

\subsubsection{Desingularization}
\begin{theor}\label{hyperth}
Let $\ogtC_\reg$ be as above. Then the functor $\calF_\gtC$ from Theorem \ref{formth} induces a strong desingularization functor $\calF_\ogtC$ which assigns to objects of $\ogtC_\reg$ countable algebraic/formal/analytic blow up hypersequences with regular maximal element $X_\infty$. In particular, the morphism $X_\infty\to X$ is a functorial desingularization of $X$ by a single proper morphism.
\end{theor}
\begin{proof}
We act as in the proof of Theorem \ref{formth}, though this time a chart $\{X_i\}$ can be infinite. We have defined resolutions $\calF_\ogtC(X_i)=\calF_\gtC(X_i)$ in the proof of Theorem \ref{formth}. Obviously, we can saturate each blow up sequence $\calF_\ogtC(X_i)$ with trivial blow ups so that one obtains a blow up hypersequence $\calF_\ogtC(X_i)$ whose objects are parameterized by the invariants of $\calF_\ogtC$ (and so there are countably many of them). Then the hypersequences $\calF_\ogtC(X_i)$ agree on the intersections because they agree over each compact subspace in the intersections, and hence glue to a single countable hypersequence $\calF_\ogtC(X)$.
\end{proof}


\begin{thebibliography}{EGA I}
\bibitem[And]{And}
Andr\'e, M.: {\it Localisation de la lissit\'e formelle}, Manuscripta Math. {\bf 13} (1974), 297--307.

\bibitem[BM1]{BM}
Bierstone, E.; Milman, P.: {\it Canonical desingularization in characteristic zero by blowing up the maximum
strata of a local invariant}, Invent. Math. {\bf 128} (1997), no. 2, 207--302.

\bibitem[BM2]{bmfun}
Bierstone, E.; Milman, P.: {\it Functoriality in resolution of singularities}, Publ. Res. Inst. Math. Sci. {\bf
44} (2008), 609--639.

\bibitem[BMT]{bmt}
Bierstone, E.; Milman, P.; Temkin M.: {\it $\bfQ$-universal desingularization}, Asian J. of Math. {\bf 15} (2011), 229--250.

\bibitem[Con]{Con}
Conrad, B.: {\it Deligne's notes on Nagata compactifications}, J. Ramanujan Math. Soc. {\bf 22} (2007), no. 3,
205--257.

\bibitem[Duc1]{duc}
Ducros, A.: {\it Les espaces de Berkovich sont excellents}, Ann. Inst. Fourier {\bf 59} (2009), 1407--1516.

\bibitem[Duc2]{duc2}
Ducros, A.: {\it Flatness in non-Archimedean analytic geometry}, preprint,  arXiv:[1107.4259].

\bibitem[EGA]{ega}
Dieudonn\'e, J.; Grothendieck, A.: {\it \'El\'ements de g\'eom\'etrie alg\'ebrique}, Publ. Math. IHES, {\bf 4,
8, 11, 17, 20, 24, 28, 32}, (1960-7).

\bibitem[EGA I]{egaI}
Dieudonn\'e, J.; Grothendieck, A.: {\it \'El\'ements de g\'eom\'etrie alg\'ebrique, I: Le langage des schemas},
second edition, Springer, Berlin, 1971.

\bibitem[Elk]{Elk}
Elkik, R.: {\it Solution d'\'equations \`a coefficients dans un anneau hens\'elien}, Ann. Sci. Ecole Norm. Sup.
(4) {\bf 6} (1973), 553--603.

\bibitem[FEM]{FEM}
de Fernex, T.; Ein, L.; Mustata, M.: {\it Shokurov's ACC Conjecture for log canonical thresholds on smooth
varieties}, Duke Math. J. {\bf 152} (2010), 93-–114.

\bibitem[Fri]{Fri}
Frisch, J.: {\it Points  de  platitude  d'un  morphisme  d'espaces  analytiques  complexes}, Invent. Math. {\bf 4} (1967), 118--138.

\bibitem[GR]{GR}
Gabber, O.; Ramero, L.: {\it Almost ring theory}, Lecture Notes in Mathematics, 1800. Springer-Verlag, Berlin, 2003, vi+307 pp.

\bibitem[Hir]{Hir}
Hironaka, H.: {\it Resolution of singularities of an algebraic variety over a field of characteristic zero. I,
II}, Ann. of Math. {\bf 79} (1964), 109--203 and 205–-326.

\bibitem[Ill]{Ill}
Illusie, L.: {\it Complexe cotangent et d\'eformations. I}, Lecture Notes in Mathematics, Vol. 239. Springer-Verlag, Berlin-New York, 1971. xv+355 pp.

\bibitem[Ked]{Ked}
Kedlaya, K.: {\it Good formal structures for flat meromorphic connections, II: Excellent schemes},
Journal of the American Mathematical Society {\bf 24} (2011), 183--229.

\bibitem[Kol]{Kol}
Koll\'ar, J.: {\it Lectures on resolution of singularities}, Annals of Mathematics Studies, 166. Princeton
University Press, Princeton, NJ, 2007. vi+208 pp.

\bibitem[Mat]{Mat}
Matsumura, H.: {\it Commutative algebra}, Mathematics Lecture Note Series, vol. 56, The Benjamin
Cummings Publishing Company, Reading, Massachusetts, 1980, seconde edition.

\bibitem[Nic]{Nic}
Nicaise, J.: {\it A trace formula for rigid varieties, and motivic Weil generating series for formal schemes},
Math. Ann. {\bf 343} (2009), 285--349.

\bibitem[NN]{NN}
Nishimura, J.; Nishimura, T. {\it Ideal-adic completion of Noetherian rings. II},  Algebraic geometry and
commutative algebra, Vol. II,  453--467, Kinokuniya, Tokyo, 1988.

\bibitem[Po]{Po}
Popescu, D.: {\it General N\'eron desingularization}, Nagoya Math. J. {\bf 100} (1985), 97–-126.

\bibitem[RG]{RG}
Raynaud, M.; Gruson, L.: {\it Crit\`eres de platitude et de projectivit\'e}, Inv. Math. {\bf 13} (1971), 1--89.

\bibitem[Tem1]{temdes}
Temkin, M.: {\it Desingularization of quasi-excellent schemes in characteristic zero}, Adv. Math., {\bf 219}
(2008), 488-522.

\bibitem[Tem2]{survey}
Temkin, M.: {\it Absolute desingularization in characteristic zero}, in Motivic Integration and its Interactions with Model Theory and Non-Archimedean Geometry, volume 2, Edited by R. Cluckers, J. Nicaise, and J. Sebag, London Mathematical Society Lecture Note Series {\bf 384}.

\bibitem[Tem3]{emb}
Temkin, M.: {\it Functorial desingularization over $\bfQ$: boundaries and the embedded case}, preprint,
arXiv:[0912.2570].

\bibitem[Vil]{Vil}
Villamayor, O.: {\it Constructiveness of Hironaka's resolution}, Ann. Sci. Ecole Norm. Sup. (4) {\bf 22} (1989),
no. 1, 1--32.

\bibitem[W\l]{Wl}
W\l odarczyk, J.: {\it Simple Hironaka resolution in characteristic zero}, J. Amer. Math. Soc. {\bf 18} (2005),
no. 4, 779--822 (electronic).

\end{thebibliography}
\end{document}